\documentclass{amsart}
\usepackage{amssymb,amsmath,latexsym}

\usepackage{amsfonts}
\usepackage{amscd}

\usepackage{graphicx,array}
\usepackage[all]{xy}
\numberwithin{equation}{section}

\def\ga{\mathfrak{a}}

\def\mfg{\mathfrak{g}}
\def\gh{\mathfrak{h}}

\def\gk{\mathfrak{k}}

\def\gp{\mathfrak{p}}

\def\gs{\mathfrak{s}}

\def\Ad{{\rm Ad}}

\def\trace{\mathrm{Tr}\,}

\newcommand{\res}{\mathrm{res}}

\renewcommand{\Im}{\mathop{\rm Im} }
\newcommand{\rO}{\mathrm{O}}

\def\C{\mathbb{C}}

\def\H{\mathbb{H}}

\def\N{\mathbb{N}}

\def\R{\mathbb{R}}

\def\Z{\mathbb{Z}}

\newcommand{\gL}{\Lambda }

\newcommand{\SO}{\mathrm{SO}}
\newcommand{\Sp}{\mathrm{Sp}}
\newcommand{\Spin}{\mathrm{Spin}}
\newcommand{\SU}{\mathrm{SU}}
\newcommand{\U}{\mathrm{U}}
\newcommand{\diag}{\mathrm{diag}}
\newcommand{\Tr}{\mathrm{Tr}\, }
\newcommand{\so}{\mathfrak{so}}

\def\cA{\mathcal{A}}

\def\cF{\mathcal{F}}

\def\cH{\mathcal{H}}
\def\cI{\mathcal{I}}

\def\cS{\mathcal{S}}

\newcommand{\Br}{B_r}
\newcommand{\Brc}{\overline{B_r(0)}}

\newcommand{\Exp}{\mathrm{Exp}}
\newcommand{\PW}{\mathrm{PW}}

\renewcommand{\Re}{\mathrm{Re}}

\newtheorem{theorem}[equation]{Theorem}

\newtheorem{lemma}[equation]{Lemma}

\newtheorem{corollary}[equation]{Corollary}

\newtheorem{proposition}[equation]{Proposition}

\newtheorem{remark}[equation]{Remark}

\newcommand{\fa}{\mathfrak{a}}

\newcommand{\fg}{\mathfrak{g}}
\newcommand{\fh}{\mathfrak{h}}
\newcommand{\fk}{\mathfrak{k}}

\newcommand{\fm}{\mathfrak{m}}
\newcommand{\fn}{\mathfrak{n}}

\newcommand{\fs}{\mathfrak{s}}

\newcommand{\fz}{\mathfrak{z}}

\newcommand{\rS}{\mathrm{S}}

\def\sideremark#1{\ifvmode\leavevmode\fi\vadjust{\vbox to0pt{\vss
 \hbox to 0pt{\hskip\hsize\hskip1em
\vbox{\hsize2cm\tiny\raggedright\pretolerance10000
 \noindent #1\hfill}\hss}\vbox to8pt{\vfil}\vss}}}%

\begin{document}

\title{The Paley-Wiener Theorem and Limits of Symmetric Spaces}

\author{Gestur \'{O}lafsson}
\address{Department of Mathematics, Louisiana State University, Baton Rouge,
LA 70803}
\email{olafsson@math.lsu.edu}
\thanks{The research of G. \'Olafsson was supported by NSF grant  DMS-0801010}

\author{Joseph A. Wolf}
\address{Department of Mathematics, University of California, Berkeley,
CA 94720--3840}
\email{jawolf@math.berkeley.edu}
\thanks{The research of J. A. Wolf was partially supported by NSF grant
DMS-0652840}


\date{}

\begin{abstract}
We extend the Paley--Wiener theorem for riemannian symmetric spaces
to an important class of infinite dimensional symmetric spaces.
For this we define a notion of propagation of symmetric spaces and
examine the direct (injective) limit symmetric spaces defined by
propagation.  This relies on some of our earlier work on
invariant differential operators and the action of Weyl group
invariant polynomials under restriction.
\end{abstract}
\maketitle

\section*{Introduction} \label{sec0}
\setcounter{equation}{0}
\noindent
We start with the notion of prolongation for symmetric spaces.  In
essence, a symmetric space $M_k$ is a prolongation of another, say
$M_n$, when $M_n$ sits in $M_k$ in the simplest possible way.  For
example, if $M_\ell = SU(\ell + 1)$, compact group manifold, then
$M_n$ sits in $M_k$ as an upper left hand corner.

Suppose that $M_k$ is a prolongation $M_n$ where both are of compact type or both of noncompact type. We prove surjectivity for restriction of Weyl group invariant
holomorphic functions of exponential growth $r$.  We discuss the conditions
on $r$ in a moment.  This gives a corresponding restriction result on the Fourier transform spaces and
then a sujective map $C^\infty_r(M_k)\to C^\infty_r(M_n)$.
Using results on conjugate and cut locus of compact symmetric spaces
we show that the radius of injectivity for compact symmetric spaces forming
a direct system, related by prolongation, is constant.  If $R$ is that
radius then the condition on the exponential growth size $r$ is a function of
$R$, thus constant for the direct system. This, together with the results
of \cite{OW10}, allows us to
carry the finite dimensional Paley--Wiener theorem to the limit.
See Theorems \ref{th-IsomorphismNonCompact2}, \ref{th-CknSurjective}
and \ref{th-PWcompactII} below.

The classical Paley--Wiener Theorem describes the growth of the
Fourier transform of a function $f \in C_c^\infty(\R^n)$ in terms of the
size of its support.  Helgason and Gangolli generalized it to riemannian
symmetric spaces of noncompact type, Arthur extended it to semisimple Lie
groups, van den Ban and Schlichtkrull made the extension to pseudo-riemannian
reductive symmetric spaces, and finally \'Olafsson and Schlichtkrull
worked out the corresponding result for compact riemannian symmetric spaces.
Here we extend these results to a class of infinite dimensional
riemannian symmetric spaces, the classical direct limits compact symmetric
spaces.  The main idea is to combine the results of \'Olafsson and Schlichtkrull
with Wolf's results on direct limits
$\varinjlim M_n$ of riemannian symmetric spaces and limits of the
corresponding function spaces on the $M_n$.

Of course compact support in the Paley--Wiener Theorem is
irrelevant for functions on a compact
symmetric space. There one concentrates on the radius of the support.
The Fourier transform space is interpreted as the parameter space for
spherical functions. It is linear dual space of the complex span of the
restricted roots.  When we pass to direct limits it is crucial that
these ingredients be properly normalized.  In order to do this
we introduce the notion of propagation for pairs of root systems, pairs
of groups, and pairs of symmetric spaces.

In Section \ref{sec1} we recall some basic facts concerning Paley--Wiener
theorems on Euclidean spaces and their behavior under the action of
finite symmetry groups.  In this setting we give surjectivity criteria for
restriction of Paley--Wiener spaces.

In Section \ref{sec2} we discuss the structural results, both for symmetric
spaces of compact type and of noncompact type, that we will need later.
In order to do this we recall our notion of propagation from \cite{OW10}
and examine the corresponding Weyl group invariants explicitly for each
type of root system. The key there is the main result of \cite{OW10}, which
summarizes the facts on restriction of Weyl groups for
propagation of symmetric spaces.

In Section \ref{sec3} we apply our results on Weyl group invariants to
Fourier analysis on riemannian symmetric spaces of noncompact type.  The
main result is Theorem \ref{th-ProjLimNonCompact}, the Paley--Wiener
Theorem for classical direct limits of those spaces.
As indicated earlier, a $\Z_2$ extension of the Weyl group is needed in
case of root systems of type $D$. The extension can be realized by
an automorphism $\sigma$ of the of the Dynkin diagram. We show that
there exists an automorphism $\widetilde{\sigma}$ of $G$ or a double cover
such that $d\widetilde{\sigma}|_\fa =\sigma$ and the spherical function
with spectral parameter $\lambda$ satisfies
$\varphi_\lambda (\widetilde{\sigma}(x))
=\varphi_{\sigma'(\lambda )}(x)$.

In Section \ref{sec4} we set up the basic surjectivity of the direct
limit Paley--Wiener Theorem for the classical sequences $\{SU(n)\}$,
$\{SO(2n)\}$, $\{SO(2n+1)\}$ and $\{Sp(2n)\}$.  The key tool is
Theorem \ref{inj-radius}, the calculation of the injectivity radius.
That radius
turns out to be a simple constant ($\sqrt{2}\,\pi$ or $2\pi$) for each
of the series.  The main result is Theorem \ref{projsys}, which
sets up the projective systems of functions used in the Paley--Wiener
Theorem for $SU(\infty)$, $SO(\infty)$ and $Sp(\infty)$.  All this is needed
when we go to limits of symmetric spaces.

In Section \ref{sec5} we examine limits of spherical representations of
compact symmetric spaces.  Theorem \ref{l-inductiveSystemOfRep} is
the main result.  It
sets up the sequence of function spaces corresponding to a direct system
$\{M_n\}$ of compact riemannian symmetric spaces in which $M_k$ propagates
$M_n$ for $k \geqq n$. We use this in Section \ref{sec6} to show that
a certain surjective map $Q: C^\infty (G)^G\to C^\infty (G/K)^K$
is in fact surjective as a map $C^\infty_r(G)^G \to C_r^\infty (G/K)^K$.
Here $Q(f)(xK):=\int_K f(xk)\, dk$ and the subscript ${}_r$ denotes the
size of the support.

Then in Section \ref{sec6}, we relate the spherical Fourier transforms
for the sequence $\{M_n\}$, show how the injectivity radii remain constant
on the sequence.  We then prove the Paley--Wiener Theorem \ref{t: PW}
for compact symmetric spaces in a form that is applicable to direct limits
$M_\infty = \varinjlim M_n$ of compact riemannian symmetric spaces in which
$M_k$ propagates $M_n$ for $k \geqq n$.  Along the way we obtain a stronger
form, Theorem \ref{stronger}, of one of the key ingredients in the proof
of the surjectivity.

Finally in Section \ref{sec7} we introduce and discuss a
$K$--invariant domain in $M$ that behaves well under propagation.
This leads to a corresponding restriction theorem,
Theorem \ref{th-PWcompactII}, and another result of Paley--Wiener type,
Theorem \ref{th-ProjLimCompact}.

Our discussion of direct limit Paley--Wiener Theorems involves function
space maps that have a somewhat indirect relation \cite{W2009}
to the $L^2$ theory of \cite{W2010}.  This is
discussed in Section \ref{sec8}, where we compare our maps with the
partial isometries of \cite{W2010}.

\section{Polynomial Invariants and Restriction of
Paley-Wiener spaces}\label{sec1}\setcounter{equation}{0}

\noindent
In this section we recall and refine some results of Cowling and Rais
that will be used later in this article.

Let $E\cong \R^n$ be a finite
dimensional Euclidean space.  Let $\langle x,y\rangle_E =
\langle x,y\rangle =x\cdot y$ denote the inner
product on $E$ and its $\C$--bilinear extension to the complexification
$E_\C\cong \C^n$.
Let $|\, \cdot \, |$ denote the corresponding norm on $E$ and $E_\C$. Note that $\langle \,\cdot \, ,\, \cdot \, \rangle$ defines an bilinear form and a norm on $E^*$ and $E^*_\C$.

Denote by
$C_r^\infty ( E )$ the space of smooth functions on
$E$ with support in a closed ball $\Brc$ of radius $r>0$.  Write
$\PW_r(E_\C^*)$ for the space of holomorphic function on $E_\C^*$ with
the property that for each $n\in\Z^+$ there exists
a constant $C_n>0$ such that
\begin{equation}\label{eq-DefPW}
\nu_n(F):=\sup_{\lambda \in E_\C} (1+|\lambda |^2)^{n}e^{-r|\Im \lambda |}|F(\lambda )|<\infty\, .
\end{equation}

Consider a $G$-module $V$.  The action on functions is given as usual by
$L_wf(v):= f(w^{-1}v)$ and we denote the fixed point set by
\begin{equation}\label{eq-invariants}
V^G=\{v\in V\mid g\cdot v=v\text{ for all } g\in G\}\, .
\end{equation}
In particular, given a closed subgroup $G \subset O(E)$, the spaces
$\PW_r(E_\C^*)^G$ and $C_r^\infty (E)^G$ are well defined.
We normalize the Fourier transform on $E$ as
\begin{equation}\label{eq-FourierTransform}
\cF_E(f)(\lambda )=\widehat{f}(\lambda )
=(2\pi )^{- n/2}\int_E f(x)e^{-i\lambda ( x )}\, dx\, , \quad \lambda \in E_\C^*\text{ and } n = \dim E.
\end{equation}
The Paley--Wiener Theorem says that
$\mathcal{F}_E :C_r^\infty (E)^G\to \PW_r(E_\C^*)^G$
is an isomorphism.

{}From now on we assume that $F$ is another
Euclidean space and that $E\subseteqq F$.
We always assume that the inner products on $E$ and
$F$ are chosen so that $\langle x,y\rangle_E=\langle x,y\rangle_F$
for all $x,y\in E$. Furthermore, if
$W(E)$ and $W(F)$ are closed subgroups of the respective orthogonal groups
acting on $E$ and $F$, then set
\[W_E(F)=\{w\in W(F)\mid w(E)=E\}\, .\]
We always assume that
$W(E)$ and $W(F)$ are generated by reflections $s_\alpha : v\mapsto
v-\frac{2\alpha (v)}{\langle \alpha ,\alpha \rangle }h_\alpha$, for $\alpha$ in a root system
in $E^*$ (respectively $F^*$). However the Cowling result
below holds for arbitrary closed subgroup of $\mathrm{O} (E)$ (respectively $\mathrm{O} (F)$).

\begin{theorem}[Cowling]\label{th-cowling} The restriction map
$\PW_r(F_\C^*)^{W_E(F)} \to \PW_r(E_\C^*)^{W_E (F)|_{E_\C}}$, given by
$F\mapsto F|_{E_\C^*}$, is surjective.
\end{theorem}

Denote by $\rS (E)$ the symmetric algebra of $E$.  It can be identified with the algebra of polynomial functions on $E^*$. We use similar notation for $F^*$.

\begin{theorem}[Rais]\label{th-rais}
Let $P_1,\ldots ,P_n$ be a basis for $\rS (F)$ over
$\rS(F)^{W(F)}$. If $F\in\PW_r(F_\C^*)$
there exist  $\Phi_1,\ldots ,\Phi_n\in\PW_r(F_\C^*)^{W(F)}$ such
that
$$F=P_1\Phi_1+ \ldots + P_n\Phi_n\, .$$
\end{theorem}

If $W_E(F)|_E=W(E)$ then Cowling's Theorem implies that the restriction map
\[\PW_r(F_\C^*)^{W_E(F)} \to \PW_r(E_\C^*)^{W (E)}\, ,\quad F\mapsto F|_{E_\C^*}\, ,\]
is surjective, but in general $\PW_r(F_\C^*)^{W(F)}$ is smaller than $\PW_r(F_\C^*)^{W_E(F)}$, so one would in general not expect the restriction map to remain surjective. The following
theorem gives a sufficient condition for that to happen.

\begin{theorem}\label{th-IsoPW1} Let the notation be as above. Assume that
$W_E(F)|_E =W(E)$ and that the restriction map $\rS(F)^{W(F)} \to \rS(E)^{W(E)}$ is surjective.
Then the restriction map
$$
\PW_r(F_\C^*)^{W(F)}\to \PW_r(E_\C^*)^{W(E)}\, \text{, given by }
F\mapsto F|_{E_\C^*}\, ,
$$
is surjective.
\end{theorem}

\begin{proof} It is clear that if $F\in\PW_r(F_\C^*)^{W(F)}$ then
$F|_{E_\C^*}\in \PW_r(E_\C^*)^{W(E)}$.
For the surjectivity let $G\in \PW_r(E_\C^*)^{W(E)}$. By Theorem \ref{th-cowling} and
our assumption on the reflection groups there
exists a function $\widetilde{G}\in \PW_r(F_\C^*)^{W_E(F)}$ such that
$\widetilde{G}|_{E_\C^*}=G$. By Theorem \ref{th-rais}, there exist
$\Phi_1,\ldots ,\Phi_n\in \PW_r(F_\C^*)^{W(F)}$ and
polynomials $P_1,\ldots ,P_n\in \rS (F)$ such that
$\widetilde G =P_1\Phi_1+\ldots + P_n\Phi_n$ and
$G=\widetilde{G}|_{E_\C^*}=(P_1|_{E_\C^*})(\Phi_1|_{E_\C^*})+\ldots + (P_n|_{E_\C^*})(\Phi_n|_{E_\C^*})$.
As $W (E) =W_E(F)|_E$, $G$ is $W (E)$--invariant and the functions $\Phi_j$ are $W(F)$--invariant,
we can average the polynomials $P_j$ over $W_E(F)$ and thus assume that $P_j|_{E_\C^*} \in \rS (E)^{W (E)}$.
But then there exists $Q_j\in \rS (F)^{W(F)}$ such that $Q_j|_{E_\C^*}=P_j|_{E_\C^*}$. Let
$\Phi:=Q_1\Phi_1+\ldots +Q_r\Phi_r$. Then $\Phi\in \PW_r(F_\C^*)^{W(F)}$ and $\Phi|_{E_\C^*}=G$.
Hence the restriction map is surjective.
\end{proof}

Let $n=\dim E$ and $m=\dim F$. Denote by $\cF_E$ respectively $\cF_F$ the
Euclidean Fourier transforms on $E$ and $F$.
The following map $C$ was denoted by $P$ in \cite{cowling}.

\begin{corollary}[Cowling]\label{co-Cowling} Let the assumptions be as above.
Then the  map
$$
C: C_r^\infty (F)^{W(F)}\to C_r^\infty (E)^{W(E)}\, ,\text{ given by }
C f(x)=\int_{E^\perp} f(x,y)\, dy,
$$
is surjective.
\end{corollary}
\begin{proof} Let $c=(2\pi )^{(n-m)/2}$. For $g\in C_r^\infty (E)^{W(E)}$ let
$G=\cF_E(g)\in \PW_r(E_\C^*)^{W(E)}$. Choice  $F\in \PW_r(F_\C^*)^{W(F)}$
such that $F|_{E_\C^*}=c^{-1} G$. With $f:=\cF_F^{-1}(G|_F)\in C_r^\infty (F)^{W(E)}$ a simple calculation shows that $C(f)=g$.
\end{proof}

\begin{theorem}\label{th-projectiveLimitPW}
Let $\{E_j\}$ be a sequence of Euclidean spaces,
$E_j\subseteqq E_{j+1}$, that satisfies the hypotheses
of {\rm Theorem \ref{th-IsoPW1}} for
each pair $(E_j,E_k)$, $k\geqq j$.  Denote the restriction maps by
$P_{j}^k :\PW_r(E_{k,\C}^*)^{W(E_{k})}\to
\PW_r (E_{j,\C}^*)^{W(E_j)}$. Then
$\{\PW_r (E_{j,\C}^*)^{W(E_j)}, P^k_j\}$ is a projective system
whose limit
$P^\infty_n :\varprojlim  \PW_r (E_{j,\C}^*)^{W(E_j)}\to \PW_r(E_{n,\C}^*)^{W (E_n)}$ is surjective for all $n$. In particular,
$\varprojlim  \PW_r (E_{j,\C}^*)^{W(E_j)} \not= \{0\}$.
\end{theorem}
\begin{proof} It is clear that $\{\PW_r (E_{j,\C}^*)^{W(E_j)}, P_{j}^k\}$
is a projective system.  Given $n$ and a nonzero
$F\in \PW_r (E_{n,\C}^*)^{W(E_n)}$, recursively
choose $F_k \in \PW_r(E_{k,\C}^*)^{W(E_k)}$ for $k\geqq n$ such that
$F_{k+1}|_{E_{k,\C}^*}=F_k$.  Then the sequence $\{F_k\}$ is a non-zero element of
$\varprojlim \PW_r (E_{j,\C}^*)^{W(E_j)}$ and $P^\infty_n (\{F_k\})=F$.
\end{proof}

\begin{theorem} Given the conditions of
{\rm Theorem \ref{th-projectiveLimitPW}} define
$C_{j}^k :C^\infty_r(E_{k})^{W(E_{k})}\to C^\infty_r (E_j)^{W(E_j)}$ by
\[[C_{j}^k(f)](x)=\int_{E_j^\perp} f(x,y)\, dy\, .\]
Then the maps $C_{j}^k$ are surjective,
$\{C^\infty_r (E_j)^{W(E_j)}, C^k_j\}$ is a projective system, and its limit
$C^\infty_n :\varprojlim C^\infty_r (E_j)^{W(E_j)}\to C^\infty_r (E_n)^{W(E_n)}$
is surjective for all $n$. In particular, $\varprojlim C^\infty_r (E_j)^{W(E_j)}\not= \{0\}$.
\end{theorem}

\begin{proof} The proof is the same as that of Theorem
\ref{th-projectiveLimitPW}, making use of Corollary \ref{co-Cowling}.
\end{proof}

\begin{remark} The last two theorems remain valid if the assumptions holds for
a cofinite subsequence of $\{E_j\}_{j\in J}$. \hfill $\diamondsuit$
\end{remark}

The elements in $\varprojlim \PW_r (E_{j,\C}^*)^{W(E_j)}$ can be viewed as
functions on the injective limit $E_\infty^*=\varinjlim E_j^*=\bigcup E_j^*$.
To see that let $F\in \varprojlim \PW_r (E_{j,\C}^*)^{W(E_j)}$ and
$v\in E_\infty$. Let $n$ be such that $v\in E_n$ and define
$F(v):=P^\infty_n (F)(v)$. The definition is clearly independent of $n$.
 Finally, as the Fourier transform
$\cF_j : C^\infty_r (E_j)^{W(E_j)}\to \PW_r (E_{j,\C}^*)^{W(E_j)}$ is an
isomorphism on each level and for
$k\geqq n$ $\cF_n\circ C^k_n = P^k_n\circ \cF_k$, we get the following theorem:

\begin{theorem} There exists an unique isomorphism
\[\cF_\infty : \varprojlim C^\infty_r (E_j)^{W(E_j)}\to \varprojlim \PW_r (E_{j,\C}^*)^{W(E_j)}\]
such that for all $n$ we have $\cF_n \circ C^\infty_n=P^\infty_n\circ \cF_\infty$.
\end{theorem}

\section{Symmetric Spaces}\label{sec2}
\noindent
In this section we apply the results of Section \ref{sec1} to harmonic
analysis on symmetric spaces of noncompact type. We start with some
general considerations that are valid for symmetric spaces both of
compact and noncompact type.

Let $M=G/K$ be a riemannian symmetric space of compact or noncompact
type. Thus $G$ is a connected semisimple Lie group with an involution
$\theta$ such that
\[(G^\theta)_o\subseteqq K\subseteqq G^\theta\]
where $G^\theta =\{x\in G\mid \theta (x)=x\}$ and the subscript ${}_o$ denotes
the connected component containing the identity element. If
$G$ is simply connected then $G^\theta$ is connected and
$K=G^\theta$. If $G$ is without compact factors and with finite
center, then $K\subset G$ is a \textit{maximal
compact} subgroup of $G$, $K$ is connected, and $G/K$ is simply
connected.

Denote the Lie algebra of $G$ by $\fg$. Then
$\theta$ defines an involution $\theta : \fg\to \fg$ and
$\fg=\fk\oplus \fs$
where $\fk=\{X\in\fg\mid \theta(X)=X\}$ is the Lie algebra of $K$ and
$\fs=\{X\in \fg\mid \theta (X)=-X\}$.

Cartan Duality is a bijection between the classes of simply connected
symmetric spaces of noncompact type and of compact type. On the Lie
algebra level this isomorphism is given by
$\fg = \fk\oplus \fs \leftrightarrow \fk\oplus i\fs = \fg^d$.
We denote this bijection by $M\leftrightarrow M^d$.

Fix a maximal abelian subset $\fa\subset \fs$. For $\alpha \in\fa^*_\C$ let
\[
\fg_{\C,\alpha} =\{X\in\fg_\C \mid [H,X]=\alpha (H)X \text{ for all }
H\in \fa_\C\}\, .
\]
If $\fg_{\C,\alpha}\not=\{0\}$ then $\alpha$ is called a (restricted) root. Denote
by $\Sigma (\fg,\fa)$ the set of roots.  If $M$ is of noncompact type, then
$\Sigma (\fg,\fa)\subset \fa^*$ and
$\fg_{\C,\alpha}=\fg_\alpha +i\fg_\alpha$, where
$\fg_\alpha =\fg_{\C,\alpha}\cap \fg$. If $M$ is of compact type, then the
roots are purely imaginary on $\fa$,
$\Sigma (\fg,\fa)\subset i\fa^*$, and $\fg_{\C,\alpha}\cap \fg=\{0\}$. The
set of roots is preserved under duality, $\Sigma(\fg,\fa )=\Sigma(\fg^d,i\fa)$,
where we view those roots as $\C$--linear functionals on $\ga_\C$.

If $\alpha \in \Sigma (\fg,\fa )$ it can happen that
$\frac{1}{2}\alpha \in\Sigma (\fg,\fa)$ or $2\alpha \in\Sigma (\fg,\fa)$
(but not both). Define
\[\Sigma_{1/2}(\fg,\fa )=\{\alpha \in\Sigma(\fg,\fa ) \mid \tfrac{1}{2 }\alpha\not\in \Sigma (\fg, \fa)\}\, .\]
Then $\Sigma_{1/2}(\fg, \fa )$ is a root system in the usual sense and the Weyl
group corresponding to $\Sigma (\fg, \fa)$ is the same as the Weyl group generated
by the reflections $s_\alpha$, $\alpha \in \Sigma_{1/2}(\fg, \fa)$.
Furthermore, $M$ is
irreducible if and only if $\Sigma_{1/2}(\fg, \fa )$ is irreducible, i.e., can not be
decomposed into two mutually orthogonal root systems.

Let $\Sigma^+(\fg, \fa)\subset \Sigma (\fg, \fa)$ be a positive system and
$\Sigma^+_{1/2}(\fg,\fa )=\Sigma^+ (\fg,\fa )\cap \Sigma_{1/2} (\fg,\fa)$. Then
$\Sigma^+_{1/2}(\fg,\fa )$ is a positive system in
$\Sigma_{1/2} (\fg,\fa)$. Denote
by $\Psi_{1/2} (\fg,\fa)=\{\alpha_1,\ldots ,\alpha_r\}$, $r=\dim \fa$, the set of simple roots in $\Sigma_{1/2}^+(\fg,\fa)$. Then
$\Psi_{1/2} (\fg,\fa )$ is a basis for $\Sigma (\fg,\fa )$. We will always assume that $\Psi_{1/2}$ is not one of the exceptional root system and we number the simple roots in the following way:

\begin{equation}\label{rootorder}
\begin{aligned}
&\begin{tabular}{|c|l|c|}\hline
$\Psi_{1/2}=A_k$&
\setlength{\unitlength}{.5 mm}
\begin{picture}(155,18)
\put(5,2){\circle{2}}
\put(2,5){$\alpha_{k}$}
\put(6,2){\line(1,0){13}}
\put(24,2){\circle*{1}}
\put(27,2){\circle*{1}}
\put(30,2){\circle*{1}}
\put(34,2){\line(1,0){13}}
\put(48,2){\circle{2}}
\put(49,2){\line(1,0){23}}
\put(73,2){\circle{2}}
\put(74,2){\line(1,0){23}}
\put(98,2){\circle{2}}
\put(99,2){\line(1,0){13}}
\put(117,2){\circle*{1}}
\put(120,2){\circle*{1}}
\put(123,2){\circle*{1}}
\put(129,2){\line(1,0){13}}
\put(143,2){\circle{2}}
\put(140,5){$\alpha_1$}
\end{picture}
&$k\geqq 1$
\\
\hline
$\Psi_{1/2}=B_k$&
\setlength{\unitlength}{.5 mm}
\begin{picture}(155,18)
\put(5,2){\circle{2}}
\put(2,5){$\alpha_{k}$}
\put(6,2){\line(1,0){13}}
\put(24,2){\circle*{1}}
\put(27,2){\circle*{1}}
\put(30,2){\circle*{1}}
\put(34,2){\line(1,0){13}}
\put(48,2){\circle{2}}
\put(49,2){\line(1,0){23}}
\put(73,2){\circle{2}}
\put(74,2){\line(1,0){13}}
\put(93,2){\circle*{1}}
\put(96,2){\circle*{1}}
\put(99,2){\circle*{1}}
\put(104,2){\line(1,0){13}}
\put(118,2){\circle{2}}
\put(115,5){$\alpha_2$}
\put(119,2.5){\line(1,0){23}}
\put(119,1.5){\line(1,0){23}}
\put(143,2){\circle*{2}}
\put(140,5){$\alpha_1$}
\end{picture}
&$k\geqq 2$\\
\hline
$\Psi_{1/2}=C_k$ &
\setlength{\unitlength}{.5 mm}
\begin{picture}(155,18)
\put(5,2){\circle*{2}}
\put(2,5){$\alpha_{k}$}
\put(6,2){\line(1,0){13}}
\put(24,2){\circle*{1}}
\put(27,2){\circle*{1}}
\put(30,2){\circle*{1}}
\put(34,2){\line(1,0){13}}
\put(48,2){\circle*{2}}
\put(49,2){\line(1,0){23}}
\put(73,2){\circle*{2}}
\put(74,2){\line(1,0){13}}
\put(93,2){\circle*{1}}
\put(96,2){\circle*{1}}
\put(99,2){\circle*{1}}
\put(104,2){\line(1,0){13}}
\put(118,2){\circle*{2}}
\put(115,5){$\alpha_2$}
\put(119,2.5){\line(1,0){23}}
\put(119,1.5){\line(1,0){23}}
\put(143,2){\circle{2}}
\put(140,5){$\alpha_1$}
\end{picture}
& $k\geqq 3$
\\
\hline
$\Psi_{1/2}=D_k$ &
\setlength{\unitlength}{.5 mm}
\begin{picture}(155,20)
\put(5,9){\circle{2}}
\put(2,12){$\alpha_{k}$}
\put(6,9){\line(1,0){13}}
\put(24,9){\circle*{1}}
\put(27,9){\circle*{1}}
\put(30,9){\circle*{1}}
\put(34,9){\line(1,0){13}}
\put(48,9){\circle{2}}
\put(49,9){\line(1,0){23}}
\put(73,9){\circle{2}}
\put(74,9){\line(1,0){13}}
\put(93,9){\circle*{1}}
\put(96,9){\circle*{1}}
\put(99,9){\circle*{1}}
\put(104,9){\line(1,0){13}}
\put(118,9){\circle{2}}
\put(113,12){$\alpha_3$}
\put(119,8.5){\line(2,-1){13}}
\put(133,2){\circle{2}}
\put(136,0){$\alpha_1$}
\put(119,9.5){\line(2,1){13}}
\put(133,16){\circle{2}}
\put(136,14){$\alpha_2$}
\end{picture}
& $k\geqq 4$
\\
\hline
\end{tabular}
\end{aligned}
\end{equation}

Later on we will also need the root system $\Sigma_2 (\fg,\fa )=
\{\alpha \in \Sigma (\fg,\fa )\mid 2\alpha\not\in \Sigma (\fg,\fa )\}$.
Following the above discussion, this will only change the simple root at the
right end of the Dynkin diagram. If $\Psi_2(\fg,\fa )$
is of type $B$ the root system $\Sigma_2 (\fg,\fa )$ will be of type $C$.

The classical irreducible symmetric spaces are given by the following
table.\footnote{More detailed information is given by the Satake--Tits diagram for $M$;
see \cite{Ar1962} or \cite[pp. 530--534]{He1978}.
In that classification the case $\SU (p,1)$, $p\geqq 1$, is denoted by $AIV$,
but here it appears in $AIII$.
The case $\SO (p,q)$, $p+q$ odd, $p\geqq q>1$, is denoted by $BI$ as in
this case the Lie algebra $\fg_\C=\so (p+q,\C)$ is of type $B$.
The case $\SO (p,q)$, with $p+q$ even, $p\geqq q>1$ is denoted by $DI$ as
in this case $\fg_\C$ is of type $D$. Finally, the case $\SO (p,1)$, $p $
even, is denoted by $BII$ and $\SO (p,1)$, $p$ odd, is denoted by $DII$.}
The fifth column lists $K$ as a subgroup of the compact
real form.  The second column indicates the type of the root system
$\Sigma_{1/2}(\fg,\fa)$.

{\footnotesize
\begin{equation}\label{symmetric-case-class}
\begin{tabular}{|c|c|l|l|l|c|c|} \hline
\multicolumn{7}{| c |}
{Irreducible Riemannian Symmetric $M = G/K$, $G$ classical, $K$ connected}\\
\hline \hline
\multicolumn{1}{|c}{} & & \multicolumn{1}{c}{$G$ noncompact}&
    \multicolumn{1}{|c}{$G$ compact} &
        \multicolumn{1}{|c}{$K$} &
        \multicolumn{1}{|c}{Rank$M$} &
        \multicolumn{1}{|c|}{Dim$M$} \\ \hline \hline
$1$ & $A_j$ &$\mathrm{SL}(j,\C)$ &$\SU (j)\times \SU(j)$ & $\diag\, \SU(j)$ & $j-1$ & $j^2-1$ \\ \hline
$2$ & $B_j$&$\SO (2j+1,\C)$&  $\SO (2j+1)\times \SO (2j+1)$ & $\diag\, \SO (2j+1)$ &
    $j$ & $2j^2+j$ \\ \hline
$3$ & $D_j$&$\SO (2j,\C)$ & $\SO (2j)\times \SO (2j)$ & $\diag\, \SO(2j)$ &
    $j$ & $2j^2-j$ \\ \hline
$4$ & $C_j$&$\Sp (j,\C)$&$\Sp (j)\times \Sp (j)$ & $\diag\,\Sp (j)$ & $j$ & $2j^2+j$ \\ \hline
$5$ & $AIII$& $\SU (p,q)$&$\SU(p+q)$ & $\mathrm{S}(\U (p)\times \U ( q))$ &
    $\min(p,q)$ & $2pq$ \\ \hline
$6$ &$AI$ &$\mathrm{SL}(j,\R)$& $\SU (j)$ & $\SO (j)$ & $j-1$ & $\tfrac{(j-1)(j+2)}{2}$ \\ \hline
$7$ &$AII$ &$\SU^*(2j)$& $\SU (2j)$ & $\Sp (j)$ & $j-1$ & $2j^2-j-1$  \\ \hline
$8$ &$BDI$&$\SO_o (p,q)$ &$\SO (p+q)$ & $\SO (p) \times \SO (q)$ &
    $\min(p,q)$ & $pq$  \\ \hline
$9$ &$DIII$&$\SO^*(2j)$ &$\SO (2j)$ & $\U (j)$ & $[\tfrac{j}{2}]$ & $j(j-1)$ \\ \hline
$10$ &$CII$&$\Sp(p,q)$ &$\Sp (p+q)$ & $\Sp (p) \times \Sp (q)$ &
    $\min(p,q)$ & $4pq$  \\ \hline
$11$ & $CI$& $\Sp (j,\R)$  &$\Sp (j)$ & $\U (j)$ & $j$ & $j(j+1)$  \\ \hline
\end{tabular}
\end{equation}
}

Only in the following cases do we have
$\Sigma_{1/2}(\fg,\fa )\not= \Sigma (\fg,\fa)$:
\begin{itemize}
\item $AIII$ for $1 \leqq p < q$,
\item $CII$ for $1 \leqq p < q$, and
\item $DIII$ for $j$ odd.
\end{itemize}
In those three cases there is exactly one simple root with
$2\alpha \in\Sigma (\fg,\fa )$
and this simple root is  at the
right end of the Dynkin diagram for $\Psi_{1/2} (\fg,\fa )$. Also, either
$\Psi_{1/2} (\fg,\fa )=\{\alpha\}$ contains one simple root or
$\Psi_{1/2}(\fg,\fa )$ is of type $B_r$
where $r=\dim \fa$ is the rank of $M$.

Finally, the only two cases where $\Psi_{1/2} (\fg,\fa )$ is of type $D$ are
the case $\SO (2j,\C)/\SO(2j)$ or  the split case
$\SO_o(p,p)/\SO (p)\times \SO (p)$.

Let $M_k=G_k/K_k$ and $M_n=G_n/K_n$ be irreducible symmetric spaces,
both of compact type or both of noncompact type. We write $\Sigma_n$,
$\Sigma_n^+$ and $W_n$ for
$\Sigma (\fg_n,\fa_n)$, $\Sigma^+ (\fg_n,\fa_n)$
and $W (\fg_n,\fa_n)$. We say that
$M_k$ \textit{propagates} $M_n$, if $G_n\subseteqq G_k$, $K_n=K_k\cap G_n$,
and either $\fa_k=\fa_n$ or choosing $\fa_n\subseteqq \fa_k$ we
only add simple roots to the left end
of the Dynkin diagram for $\Psi_{n, 1/2}$ to obtain the Dynkin diagram
for $\Psi_{k, 1/2}$.
So, in particular $\Psi_{n, 1/2}$ and
$\Psi_{k, 1/2}$ are of the same type.
In general, if
$M_k$ and $M_n$ are riemannian symmetric spaces
of compact or noncompact type, with universal covering
$\widetilde{M_k}$ respectively $\widetilde{M_n}$, then $M_k$
\textit{propagates}
$M_n$ if we can enumerate the irreducible factors of $\widetilde{M}_k= M_k^1\times
\ldots \times M_k^j$ and $\widetilde{M}_n=M_n^1\times \ldots \times M_n^i$, $i \leqq j$
so that $M_k^s$ propagates $M_n^s$ for $s=1,\ldots ,i$.
Thus, each $M_n$ is, up to covering, a product of irreducible factors
listed in Table \ref{symmetric-case-class}.

In general we can construct infinite sequences of propagations by moving
along each row in Table \ref{symmetric-case-class}. But there are also  inclusions like
$\mathrm{SL} (n,\R )/\SO (n)\subset \mathrm{SL} (k,\C)/\SU (k)$ which
satisfy the definition of propagation.

When $\fg_k$ propagates $\fg_n$, and $\theta_k$ and
$\theta_n$ are the corresponding involutions with
$\theta_k|_{\fg_n} = \theta_n$, the corresponding eigenspace decompositions
$\fg_k=\fk_k\oplus \fs_k$ and $\fg_n=\fk_n\oplus \fs_n$
give us
\[
\fk_n=\fk_k\cap \fg_n\, ,\quad \text{and}\quad \fs_n=\fg_n\cap \fs_k\, .\]
We recursively choose maximal commutative subspaces $\fa_k\subset \fs_k$ such
that $\fa_{n} \subseteqq \fa_k$ for $k\geqq n$. Assume for the moment that $M_j$ is irreducible. Define an extended Weyl group $\widetilde{W}_n=\widetilde{W}(\fg_n,\fa_n)$ in the following way. If $\Psi_{n,1/2}$ is not of type $D$ then $\widetilde{W}_n=W_n$. If $\Psi_{n,1/2}$ is of type $D$, then $W_n$ is the group of permutations of $\{1,\ldots ,r_n \}$, $r_n=\dim \fa_n$, and \textit{even} number of sign changes. Let $\widetilde{W}_n$ be the extension of $W_n$ by allowing all sign changes. $\widetilde{W}_n$  can be written as $W_n\rtimes \{1,\sigma\}$ where $\sigma$ corresponds to the involution on the Dynkin diagram given by $\sigma (\alpha_1)=\alpha_2$, $\sigma (\alpha_2)=\alpha_1$ and $\sigma (\alpha_i)=\alpha_i$ for $i\geqq 3$. We note that $\widetilde{W}_n$ is isomorphic to the Weyl group generated by a root system of type $B$ and hence a finite reflection group. For general symmetric spaces we define $\widetilde{W}_n$ as the product of the $\widetilde{W}$s for each irreducible factor.  Let $k\geqq n$. As before we let
\begin{equation}\label{eq-WknA}
\widetilde{W}_{k,\fa_n}=\widetilde{W}_{\fa_n}(\fg_k,\fa_k):=\{w\in
\widetilde{W}_k \mid w(\fa_n)=\fa_n\}\, .
\end{equation}

Without loss of generality,
if $\Psi_{n, 1/2}$ is of type $D$ we only consider
propagation for $r_k\geqq r_n\geqq 4$. As we only add simple roots at the
left end and those roots are orthogonal to $\alpha_1$ and $\alpha_2$ and
fixed by $\sigma_k$  it follows that $\sigma_k|_{\fa_n}=\sigma_n$.

\begin{theorem}\label{th-AdmExtG/K} Assume that $M_k$ and $M_n$ are
symmetric spaces of compact or noncompact type and that $M_k$ propagates
$M_n$. Then
\[
W_{\fa_n}(\fg_k,\fa_k)|_{\fa_n}=\widetilde{W}_{\fa_n}(\fg_k,\fa_k)|_{\fa_n} = \widetilde{W}(\fg_n,\fa_n)\]
and the restriction maps are surjective:
\[\rS (\fa_k)^{W_k}|_{\fa_n}=\rS(\fa_k)^{\widetilde{W}_k}|_{\fa_n} = \rS (\fa_n)^{\widetilde{W}_n}\, .\]
\end{theorem}

\begin{proof} The proof is a case by case inspection of the classical root systems, see \cite{OW10}.
\end{proof}

\section{Application to Fourier Analysis on Symmetric Spaces of the noncompact
Type}\label{sec3}
\noindent
In this section we apply the above results to harmonic
analysis.
We first recall the main ingredients for the Helgason Fourier transform
on a riemannian symmetric space $M=G/K$ of the noncompact type. The material
is standard and we refer to \cite{He1984} for details.  Retain the
notation of the previous section: $\Sigma (\fg,\fa )$
is the set of (restricted) roots of $\fa$ in $\fg$ and
$\Sigma^+(\fg,\fa )\subset \Sigma (\fg,\fa )$ is a positive system. Let
$$
\fn = \bigoplus_{\alpha \in \Sigma^+(\fg,\fa )} \fg_\alpha, \ \ \
\fm = \fz_{\fk}(\fa), \ \ \text{ and } \ \
\gp = \fm + \fa + \fn.
$$
Denote by $N$ (respectively $A$) the analytic subgroup of
$G$ with Lie algebra $\fn$ (respectively $\fa$). Let
$M=Z_K(\fa)$ and $P=MAN$. Then $M$ and $P$ are
closed subgroup of $G$ and $P$ is a \textit{minimal parabolic subgroup}.
Note, that we are using $M$ in two different ways, once as the
symmetric space $M$ and also as a subgroup of $G$. The meaning
will always be clear from the context.

We have the Iwasawa decomposition
\[G=KAN :\ C^\omega\text{--diffeomorphic to } K\times A\times N \text{ under }
(k,a,n) \mapsto kan\, .\]
For $x\in G$ define $k(x) \in K$ and $a(x)\in A$  by $x \in k(x)a(x)N$.
For $a \in A$ define $\log(a) \in \ga$ by $a = \exp(\log(a))$.
Then $x \mapsto k(x)$ and $x\mapsto a(x)$ are analytic.
For $\lambda\in \fa_\C^*$ let $a^\lambda :=e^{\lambda (\log (a))}$.
Then
\[man\mapsto \chi_\lambda (man):=a^\lambda \]
defines a character $\chi_\lambda$ of the group $P$, and
$\chi_\lambda$ is unitary if and only if
$\lambda \in i\fa^*$. Let $m_\alpha =\dim \fg_\alpha$ and
\[\rho = \tfrac{1}{2}\sum_{\alpha\in\Sigma^+(\fg,\fa)}m_\alpha \, \alpha\, .\]
Denote by $\pi_\lambda$ the representation of $G$ induced from $\chi_\lambda$.
It can be realized as acting on $L^2(K/M)$ by
\[\pi_\lambda (x)f(kM)=a(x^{-1}k)^{-\lambda-\rho}f(k(x^{-1}k)M)\, .\]
The constant function $\mathbf{1} (kM)=1$ is a $K$-fixed vector and the corresponding spherical function is
\begin{equation}\label{def-spherical}
\varphi_\lambda (x)=(\pi_\lambda (x)\mathbf{1},\mathbf{1})=\int_K a(x^{-1}k)^{-\lambda -\rho}\, dk
=\int_K a(xk)^{\lambda -\rho}\, dk
\end{equation}
where the Haar measure $dk$ on $K$ is normalized by $\int_K\, dk=1$.
We have $\varphi_\lambda
=\varphi_\mu$ if and only if $\mu\in W (\fg,\fa) \cdot \lambda$,
and every spherical function on $G$  is equal to some $\varphi_\lambda$.

The \textit{spherical Fourier transform} on $M$ is given by
\[\cF(f)(\lambda)=\widehat{f}(\lambda ):=\int_M f(x)\varphi_{-\lambda }(x)
\, dx\, \quad f\in C_c^\infty (M)^K\, .\]
The invariant measure $dx$ on $M$ can be normalized so that
the spherical Fourier transform extends to an unitary isomorphism
\[f\mapsto \widehat{f}\, ,\quad L^2(M)^K\cong
L^2\left(i\fa^*, \tfrac{d\lambda}{\# W  |c(\lambda )|^2}\right)^{W}\]
where
$c(\lambda )$ denotes the Harish-Chandra $c$--function.
For $f\in C_c^\infty (M)^K $ the inversion is given by
\[f(x)=\frac{1}{\# W } \int_{i\fa^*} \widehat{f}(\lambda )\varphi_{\lambda }
(x)\frac{d\lambda }{|c (\lambda )|^2}\, .\]

Recall the involution $\sigma$ on $\fa$ (and $\fa^*$)  that corresponds to the non-trivial involution of the Dynkin diagram defined above in case $\Psi_{1/2}$ is of type $D$.

\begin{lemma}\label{le-AutDynkDiagrG} Let $M$ be one of the irreducible symmetric spaces of type $D$. Then there exists an involution $\tilde{\sigma }: G\to G$ such that
\begin{enumerate}
\item $\widetilde{\sigma}|_\fa=\sigma$ where by abuse of notation we write
$\widetilde{\sigma}$ for $d\widetilde{\sigma}$,
\item $\widetilde{\sigma}$ commutes with the
the Cartan involution $\theta$, and in particular
$\widetilde{\sigma}(K)=K$,
\item $\widetilde{\sigma}(N)=N$.
\end{enumerate}
\end{lemma}

\begin{proof} One can prove this using a Weyl basis for $\fg_\C$ (see, for example, \cite[page 285]{V1974}). But the simplest proof is to note that we can replace $\SO (2j,\C)/SO (2j)$ by $\rO (2j,\C)/\rO (2j)$. Take
\[\fa =\left\{\left.\begin{pmatrix} t_1 X & & \\
&\ddots & \\
& & t_n X\end{pmatrix}\, \right| t_1,\ldots ,t_n\in\R\right\}\text{ where } X=
i\begin{pmatrix} 0 & 1\\ -1 & 0\end{pmatrix},\] and then
then $\widetilde{\sigma}$ is conjugation by $\diag (1,\ldots ,1, -1)$. Similar construction can also be done for the other case $\SO_o(p,p)/\SO (p)\times \SO (p)$ by replacing $\SO_o (p,p)$ by $\rO (p,p)$.
\end{proof}

In the general case we let $\widetilde \sigma$ be the identity on factors not of type $D$ and the above constructed involution $\widetilde\sigma$ on factors of type $D$. Similar for the involution $\sigma$ on $\fa$ and $\fa^*$. We need to extend $K$ to a group $\widetilde{K}$ acting on $M$. In case the irreducible factor is not of type $D$ then the corresponding $\widetilde K $-factor is just $K$ and otherwise $K\rtimes \{1,\widetilde{\sigma}\}$. Note that $\widetilde{W}(\fg,\fa)=N_{\widetilde{K}}(A)/Z_{\widetilde{K}}(A)$.

\begin{theorem}\label{th-SphericalFctTildeWInv} We have
$\varphi_{\lambda}(\widetilde \sigma (x))=\varphi_{\sigma (\lambda )}(x)$
and $\cF(f\circ \widetilde{\sigma}) (\lambda )=\cF (f)(\sigma (\lambda ))$
whenever $f\in C_c(M)^K$. In particular,  $f\in C_c(M)^{\widetilde K}$ if
and only if $\cF (f)$ is $\sigma$-invariant.
\end{theorem}
\begin{proof} This follows from
\begin{eqnarray*}
k(\widetilde{\sigma}(x))a(\widetilde{\sigma}(x))n(\widetilde{\sigma}(x))
&=&\sigma (x) = \widetilde{\sigma}(k(x)a(x)n(x))\\
&=& \widetilde{\sigma}(k(x))\widetilde{\sigma}(a(x))\widetilde{\sigma}(n(x))
\end{eqnarray*}
and hence $a(\widetilde{\sigma}(x))=\widetilde \sigma (a(x))$. The claim for the spherical function $\varphi_\lambda$ follows now from the integral formula (\ref{def-spherical}). That
$\cF(f\circ \widetilde{\sigma}) (\lambda )=\cF (f)(\sigma (\lambda ))$ follows from the invariance of the invariant measure on $M$ under $\widetilde{\sigma}$. The last statements follows then from the fact that the Fourier transform is injective on $C_c^\infty (M)^K$.
\end{proof}

Fix a positive definite $K$--invariant bilinear form
$\langle \cdot ,\cdot \rangle $ on $\fs$. It defines an
invariant riemannian structure on $M$ and hence an invariant
metric $d(x,y)$. Let $x_o=eK\in M$ and for $r>0$ denote by $B_r=B_r(x_o)$
the closed ball
\[B_r =\{x\in M\mid d(x,x_o)\leqq r\}\, .\]
Note that $B_r$ is $\widetilde K$--invariant. Denote by $C_r^\infty (M)^{\widetilde K}$ the space of
smooth $\widetilde{K}$--invariant functions on $M$ with support in $B_r$.
The restriction map $f\mapsto f|_A$ is a bijection from
$C_r^\infty (M)^{\widetilde{K}}$ onto $C_r^\infty (A)^{\widetilde W}$ (using the
obvious notation).

The following is a simple modification
of the Paley-Wiener theorem of Helgason \cite{He1966,He1984} and Gangolli
\cite{Ga1971}; see \cite{OP2004} for a short overview.

\begin{theorem}[The Paley-Wiener Theorem]\label{th-IsomorphismNonCompact1} The Fourier
transform defines  bijections
$$
 C^\infty_r (M)^K \cong \PW_r(\fa_\C^*)^{W } \text{ and }
C_{r}^\infty (M)^{\widetilde{K}} \cong
\PW_r(\fa_\C^* )^{\widetilde{W}}\, .
$$
\end{theorem}
\begin{proof} This follows from  the Helgason-Gangolli Paley-Wiener theorem and
Theorem \ref{th-SphericalFctTildeWInv}.
\end{proof}

We assume now that $M_k$ propagates $M_n$, $k\geqq n$. The index $j$ refers
to the symmetric space $M_j$, for a function $F$ on $\fa_{k,\C}^*$ let $P_{n}^k(F):=
F|_{\fa_{n,\C}^*}$. We fix a compatible $K$--invariant
inner products on $\fs_n$ and $\fs_k$, i.e.,
$\langle X,Y\rangle_k=\langle X,Y\rangle_n$
for all $X,Y\in\fs_n\subseteqq \fs_k$.
\begin{theorem}[Paley-Wiener Isomorphisms]\label{th-IsomorphismNonCompact2}
Assume that $M_k$ propagates $M_n$. Let $r>0$. Then the following hold:
\begin{enumerate}
\item The map $P_{n}^k:
\PW_r (\fa_{k,\C}^*)^{\widetilde{W}_k} \to
\PW_r(\fa_{n,\C}^*)^{\widetilde{W}_n}$ is surjective.
\item The map $C_{n}^k=\cF^{-1}_n\circ P_{n}^k\circ \cF_k:C^\infty_{r}
(M_k)^{\widetilde{K_k}}
\to C^\infty_{r}(M_n)^{\widetilde{K_n}}$ is surjective.
\end{enumerate}
\end{theorem}
\begin{proof}  This
follows from Theorem \ref{th-IsoPW1}, Theorem  \ref{th-AdmExtG/K} and Theorem \ref{th-IsomorphismNonCompact1} as $\widetilde{W}$ is a finite reflection group.
\end{proof}

We assume now that $\{M_n,\iota_{k,n}\}$ is a injective system of symmetric spaces
such that $M_k$ propagates $M_n$. Here $\iota_{k,n}: M_n\to M_k$ is the
injection. Let
\[M_\infty =\varinjlim M_n\, .\]

We have also, in a natural way, injective systems
$\fg_n\hookrightarrow \fg_k$, $\fk_n\hookrightarrow \fk_k$, $\fs_n\hookrightarrow \fs_k$,
and $\fa_n\hookrightarrow \fa_k$ giving rise to corresponding injective systems. Let
$$
\mfg_\infty :=\varinjlim \mfg_n\, ,\quad \gk_\infty :=\varinjlim \gk_n\, ,\quad
\gs_\infty :=\varinjlim\gs_n\, ,  \quad\text{and} \quad  \quad \ga_\infty := \varinjlim\ga_n .$$
Then $\mfg_\infty=\gk_\infty \oplus \gs_\infty$ is the eigenspace decomposition of
$\mfg_\infty $ with respect to the involution $\theta_\infty :=\varinjlim \theta_n$,
$\fa_\infty$ is a maximal abelian subspace of $\fs_\infty$.

The restriction maps $\res_{n}^k : \rS (\fa_k)^{\widetilde{W}_k}\to \rS (\fa_n)^{\widetilde{W}_n}$ and the
maps from Theorem \ref{th-IsomorphismNonCompact2} define projective systems
$\{\rS(\fa_n)^{\widetilde{W}_n}\}_n$,
$\{\PW_{r}(\fa_{n,\C}^*)^{\widetilde{W}_n}\}_n$, and
$\{C_{r}(M_n)^{\widetilde K_n}\}_n$.

Write $\Psi_{n,1/2}=\{\alpha_{n,1},\ldots ,\alpha_{n,r_n}\}$.
There is a canonical inclusion $\widetilde{W}_n
\stackrel{\iota_{k,n}}{\hookrightarrow}\widetilde{W}_{k,\fa_n}$ given by
$s_{\alpha_{n,j}}\mapsto s_{\alpha_{k,j}}$, $1\leqq j\leqq r_n$ and $\sigma_n\mapsto \sigma_k$. This map can also be constructed by realizing the extended Weyl groups as permutation group extended by sign changes.
We have $\iota_{k,n}(s)|_{\fa_n}=s$. In this way, we get an injective
system $\{\widetilde{W}(\fg_n,\fa_n)\}_n$.
We also have a natural injective system $\{\widetilde{K}_n\}$.
The restriction maps $\fa_{k,\C}^*\to \fa_{n,\C}^*$ lead to a projective system. Let
$\fa_{\infty,\C}^*:=\varprojlim \fa_{n,\C}^*$ and set
\begin{eqnarray*}
\widetilde W_\infty &: = &\varinjlim \widetilde{W}_n \\
\widetilde{K}_\infty &:=&\varinjlim \widetilde{K}_n\\
\rS_{\infty }(\fa_\infty )^{\widetilde{W}_\infty}&:=&\varprojlim \rS(\fa_n)^{\widetilde{W}_n} \\
\PW_r(\fa_{\infty,\C}^*)^{\widetilde{W}_\infty}&:=&\varprojlim \PW_{r}(\fa_{n,\C}^*)^{\widetilde{W}_n}\\
C_{r}^\infty (M_\infty)^{\widetilde K_\infty}
&:=&\varprojlim C^\infty_{r}(M_n)^{\widetilde K_n} \, .
\end{eqnarray*}
We can view $\rS_{\infty }(\fa_\infty )^{\widetilde{W}_\infty}$ as $\widetilde{W}_\infty$--invariant
polynomials on $\fa_{\infty,\C}^*$ and $\PW_r(\fa_{\infty,\C}^*)^{\widetilde{W}_\infty} $ as
$\widetilde{W}_\infty$--invariant functions on $\fa_{\infty,\C}^*$.
The projective limit $C_{r,\infty}^\infty (M_\infty)^{K_\infty}$ consists of
functions on on $A_\infty =\varinjlim A_n$, where $A_n=\exp \fa_n$.  In
Section \ref{sec8} we discuss a direct limit function space on $M_\infty$
that is more closely related to the representation theory of $G_\infty$.

For $\mathbf{f}=(f_n)_n\in
C_{r,\infty}^\infty (M_\infty)^{K_\infty}$ define $\cF_\infty (\mathbf{f})
\in
\PW_r(\fa^*_{\infty,\C} )^{\widetilde{W}_\infty}$ by
\begin{equation}
\cF_\infty (\mathbf{f}):=\{\cF_n(f_n)\}\, .
\end{equation}
Then $\cF_\infty (\mathbf{f})$ is well defined by Theorem \ref{th-IsomorphismNonCompact2}
and we have a commutative diagram
$$\xymatrix{\cdots & C^\infty_{r}(M_n)^{\widetilde K_n}\ar[d]_{\cF_n}&
C^\infty_{r}(M_{n+1})^{\widetilde K_{n+1}}
\ar[d]_{\cF_{n+1}}\ar[l]_{C_{n}^{n+1}}&\ar[l]_{C_{n+1}^{n+2}}
\quad\qquad \cdots &
C_{r}^\infty (M_\infty)^{\widetilde K_\infty}\ar[d]_{\cF_\infty}
 \\
\cdots & \PW_r(\fa^*_{n,\C})^{\widetilde{W}_n}&  \PW_r(\fa^*_{n+1,\C})^{\widetilde{W}_{n+1}}
 \ar[l]^{P_{n}^{n+1}}
&\ar[l]^{P_{n+1}^{n+2}} \quad\qquad \cdots & \PW_r(\fa_{\infty,\C}^* )^{\widetilde{W}_\infty}\\
}
$$
Then the maps
\[C^\infty_n :  C_{r}^\infty (M_\infty)^{\widetilde K_\infty}\to
  C_r^\infty (M_n)^{\widetilde K_n} \text{ and }
  P^\infty_n :\PW_r(\fa_{\infty,\C}^*)^{\widetilde{W}_\infty}\to
  \PW_r(\fa_{n,\C}^*)^{\widetilde W_n}\]
are well defined.

\begin{theorem}[Infinite dimensional Paley-Wiener Theorem]\label{th-ProjLimNonCompact}
Let the notation be as above. Then the projection maps
$C^\infty_n$ and $P^\infty_n$
are surjective. In
particular,  $C_{r}^\infty (M_\infty)^{\widetilde K_\infty}\not=\{0\}$ and
$\PW_r(\fa_{\infty,\C}^*)^{\widetilde{W}_\infty}\not=\{0\}$.
Furthermore,
\[\cF_\infty :C_{r}^\infty (M_\infty)^{\widetilde K_\infty}
\to \PW_r(\fa_{\infty,\C}^*)^{\widetilde{W}_\infty}\]
is a linear isomorphism.
\end{theorem}

\section{Central Functions on Compact Lie Groups}\label{sec4}
\noindent
The following results on compact Lie groups are a special case of
the more general statements on compact symmetric spaces discussed
in the next section, as every group can be viewed as a symmetric
space $G\times G/\mathrm{diag}(G)$ via the map
\[(g,1)\mathrm{diag}(G)\mapsto g, \text{ in other words }
    (a,b)\mathrm{diag}(G)\mapsto ab^{-1}\]
corresponding to the involution $\tau (a,b)=(b,a)$. The action of
$G\times G$ is the left-right action $(L\times R)(a,b)\cdot x=
axb^{-1}$ and the $\textrm{diag}(G)$--invariant functions
are the \textit{central} functions $f(axa^{-1})= f(x)$ for
all $a,x\in G$.
Thus $f$ is central if and only if $f\circ \Ad (a)=f$ for
all $a\in G$, where as usual $\Ad (a)(x)=axa^{-1}$. But it is still worth
treating this case separately, first because the normalization of the
Fourier transform on $G$ viewed as a group is different from the normalization
as a symmetric space, and second because the proof of the Paley-Wiener
Theorem for compact symmetric spaces in \cite{OS} was by reduction to
this case, as was originally done in \cite{Gon}.

In this section $G$, $G_n$ and $G_k$ will
denote  compact connected semisimple Lie groups. For simplicity, we will
assume that those groups are simply connected. For the general case one
needs to change the semi-lattice of
highest weights of irreducible representations and the injectivity
radius, whose numerical value does not play an important rule in the
following. The invariant measures on compact groups and homogeneous spaces
are normalized to total mass one.

We say that $G_k$ propagates $G_n$ if $\fg_k$ propagates $\fg_n$.
This is the same as saying that $G_k$ propagates $G_n$ as a symmetric space.
We fix a Cartan subalgebra $\fh_k$ of $\fg_k$ such that $\fh_n:=
\fh_k\cap \fg_n$ is a Cartan subalgebra of $\fg_n$. We use the notation
from the previous section. The index $n$ respectively $k$ will then
denote the corresponding object for $G_n$ respectively $G_k$. We fix
inner products $\langle \cdot ,\cdot \rangle_n$ on $\fg_n$
and $\langle \cdot ,\cdot \rangle_k$ on $\fg_k$ such that
$\langle X,Y\rangle_n=\langle X,Y\rangle_k$ for $X,Y\in \fg_n\subseteqq \fg_k$.
This can be done by viewing $G_n\subset G_k$ as locally isomorphic to
linear groups and use the trace form $X,Y\mapsto -\Tr(XY)$.
We denote by $R$ the injectivity radius. Theorem
\ref{inj-radius} below shows that the injectivity radius is the same for
$G_n$ and $G_k$.

The following is a reformulation of results of Crittenden \cite{crit}.
A case by case inspection of each of the root systems gives us
\begin{theorem}\label{inj-radius}
The injectivity radius of the classical compact simply connected Lie groups
$G$, in the riemannian metric given by the inner product
$\langle X,Y\rangle = -\trace(XY)$ on $\fg$, is $\sqrt{2}\,\pi$ for $SU(m+1)$
and $Sp(m)$, $2\pi$ for $SO(2m)$ and $SO(2m+1)$.  In particular
for each of the four classical series the injectivity radius $R$ is
independent of $m$.
\end{theorem}

Denote by  $\gL^+(G)\subset i\fh^*$ the set of dominant
integral weights,
\[\gL^+(G)=\left \{\mu\in i\fh^*)\left |
\,\, \tfrac{2 \langle \mu ,\alpha \rangle }{\langle \alpha ,\alpha\rangle }\in \Z^+ \text{ for all }
 \alpha \in \Delta^+(\fg_\C,\fh_\C)\right . \right \}\, .\]
For $\mu\in\gL^+(G)$ denote by
$\pi_\mu$ the corresponding representation with highest weight
$\mu$.  As $G$ is assumed
simply connected $\mu \mapsto \pi_\mu$, is a bijection from
$\gL^+(G)$ onto $\widehat{G}$. The representation
space for $\pi_\mu$ is denoted by $V_\mu$.  Let $\chi_\mu=\Tr \circ \pi_\mu$ be the character of $\pi_\mu$
and  $\deg(\mu )=\dim V_\mu$ its dimension. Then
$\deg (\mu )$ is a polynomial function on $\fh_\C^*$. The space
$L^2(G)^G:= \{f \in L^2(G) \mid f\circ\Ad (g) = f \text{ for all } g \in G\}$
contains the set $\{\chi_\mu\}_{\mu\in\gL^+(G)}$ of characters
as a complete orthonormal set.

For
$f\in C (G)^G$ define the Fourier transform $\cF(f)=\widehat{f}:\gL^+(G)\to \C$ by
$$\widehat f(\mu)= ( f,\chi_\mu )=\int_G f(x)\overline{\chi_\mu(x)}\,dx, \quad \mu\in\gL^+(G)\, ,$$
where $( f,\chi_\mu )$ is the inner product in $L^2(G)$.
The Fourier transform extends to an unitary isomorphism
$\cF : L^2(G)^G \to \ell^2(\gL^+(G))$ and
\[f=\sum_{\mu \in \gL^+(G)}\widehat{f}(\mu )\chi_\mu\]
in $L^2(G)^G$. If $f$ is smooth
the Fourier series converges in the topology of $C^\infty (G)^G$.

If not otherwise stated we will assume that $G$ does not contain
any simple factor of exceptional type.
As before $W(\fg,\fh)$ denotes the Weyl group of $\Delta (\fg_\C,\fh_\C)$,
and $\widetilde{W}=\widetilde{W}(\fg,\fh)$ denotes the extension of
$W(\fg,\fh)$ by
$\sigma$. Similarly, $\widetilde{K}$ and $\widetilde{G}$ denote the
extensions of $K$ and $G$, respectively, by $\widetilde{\sigma}$.
For $r>0$  let $\PW_r^\rho (\fh^*_\C)^{\widetilde{W}}$ denote the space
of holomorphic functions $\Phi$ on $\fh_\C^*$ such that
\begin{enumerate}
\item  For each $k\in\N$ there exists a constant $C_k>0$ such that
$$|\Phi(\lambda )|\leqq C_k(1+|\lambda |)^{-k} e^{r|\Re\lambda |}
\text{ for all } \lambda \in\fh_\C^*,
$$
\item $\Phi(w(\lambda +\rho)-\rho)=\det(w)\Phi(\lambda )$ for
all $w\in \widetilde{W}$, $\lambda \in\fh_\C^*$.
\end{enumerate}
Let $H=\exp (\fh)$.
For $0<r<R$ denote by $C_{r}^\infty (G)^{\widetilde G}$ the space of
smooth function on $G$ that are invariant under conjugation by
$\widetilde{G}$ and are supported in the closed geodesic ball $B_r(e)$
of radius $r$.
We have that $f\in C_{r}^\infty (G)^{\widetilde G}$ if and only if $f|_H\in C_r^\infty (H)^{\widetilde{W}}$.
In this terminology the theorem of Gonzalez \cite{Gon} reads as follows.

\begin{theorem}\label{t: Gonzalez}
Let $G$ be an arbitrary connected simply connected compact
Lie group.
Let $0<r<R$ and let $f\in C^\infty(G)^G$ be given. Then $f$ belongs to
$C^\infty_{r}(G)^{\widetilde G}$  if and only if the
Fourier transform $\mu\mapsto \widehat{f}(\mu)$
extends to a holomorphic function $\Phi_f$
on $\fh^*_\C$ such that $\Phi_f\in \PW_r^\rho(\fh^*_\C)^{\widetilde{W}}$.
\end{theorem}

\begin{proof} We only have to check that $f\in C^\infty_{r,\widetilde{W}}(G)^G$
if and only if $\widehat{f}(w(\mu +\rho)-\rho )=
\widehat{f}(\mu )$. For factors not of type $D_n$ that follows from
Gonzalez's theorem. For factors of type $D_n$ it follows Weyl's character
formula.
\end{proof}
In \cite{OS} it is shown that the extension $\Phi_f$ is unique whenever
$r$ is sufficiently small. In that case Fourier transform,
followed by holomorphic extension, is a bijection
$C^\infty_{r} (G)^{\widetilde G}\cong \PW_r^\rho(\fh^*_\C )^{\widetilde{W}}$.

We will now extend these results to projective limits. We start with two
simple lemmas.

\begin{lemma}\label{le-PhZero}
Let $\Phi\in\PW_r^\rho(\fh^*_\C)^{\widetilde{W}}$.
Assume that $\lambda\in \fh^*_\C$ is
such that $\langle \lambda ,\alpha\rangle =0$ for some $\alpha\in\Delta$. Then
$\Phi (\lambda-\rho )=0$.
\end{lemma}
\begin{proof} Let $s_\alpha $ be the reflection in the hyper plane
perpendicular to $\alpha$. Then
\begin{eqnarray*}
\Phi (\lambda -\rho )&=&\Phi (s_\alpha (\lambda )-\rho )\\
&=& \Phi (s_\alpha (\lambda -\rho +\rho )-\rho) = \det (s_\alpha )\Phi (\lambda -\rho)\, .
\end{eqnarray*}
The claim now follows as $\det (s_\alpha )=-1$.
\end{proof}

\begin{lemma}\label{isoPWspaces}
Let $r>0$ and let $\widetilde{W}$ be as before.
For $\Phi \in \PW_r^\rho (\fh^*_\C)^{\widetilde{W}}$ define
\[
T(\Phi ) (\lambda )=F_\Phi (\lambda ):=
\tfrac{\varpi (\rho) }{\varpi (\lambda )}\Phi (\lambda -\rho )
\text{ where } \varpi (\lambda )=\prod_{\alpha\in\Delta^+}
    \langle\lambda ,\alpha \rangle\, .
\]
Then $T (\Phi )\in \PW_r (\fh^*_\C)^{\widetilde{W}}$ and
$T: \PW_r^\rho (\fh^*_\C)^{\widetilde{W}}\to \PW_r(\fh^*_\C)^{\widetilde{W}}$ is a linear isomorphism.
\end{lemma}
\begin{proof} Let $\alpha \in\Delta^+$. Then
$\lambda \mapsto \tfrac{1}{(\lambda ,\alpha)} \Phi (\lambda )$
is holomorphic by Lemma \ref{le-PhZero}. According to
\cite{He1994}, Lemma 5.13 on page 288, it follows that
this function is also of exponential type $r$. Iterating this for each root it follows that
$F_\Phi$ is holomorphic of exponential type $r$. As $\varpi (w(\lambda ))=\det (w)\varpi (\lambda )$
it follows using the same arguments as in the proof of Lemma \ref{le-PhZero} that
$F_\Phi$ is $\widetilde W$--invariant. The surjectivity follow as
$F\mapsto \varpi (\lambda )F(\cdot +\rho)$ maps $\PW_r(\fh^*_\C)^{\widetilde{W}}$
into $\PW_r^\rho
(\fh^*_\C)^{\widetilde{W}}$.
\end{proof}
\begin{theorem}\label{th-PknSurjective} Let $r>0$ and assume that
$G_k$ propagates $G_n$. Then the map
\[\Phi \mapsto P^{k}_{n}(\Phi):=T_n^{-1}( T_k(\Phi)|_{\fh_{n,\C}^*})=
\frac{\varpi_n(\bullet )}{\varpi_n(\rho_n)} \left(\frac{\varpi_k (\rho_k)}{\varpi_k (\bullet )}
\Phi (\bullet -\rho_k)|_{\fh_{n,\C}^*}
\right)(\bullet +\rho_n)  \]
from $\PW_r^{\rho_k}(\fh_{k,\C}^*)^{\widetilde{W}_k}\to
\PW_r^{\rho_n}(\fh_{n,\C}^*)^{\widetilde{W}_n}$ is surjective.
\end{theorem}
\begin{proof} This follows from Lemma \ref{isoPWspaces} and Theorem \ref{th-IsoPW1}.
\end{proof}

Recall from Theorem \ref{inj-radius} that the injectivity radii
$R$ are the same for $G_k$ and $G_n$.  For $0<r<R$
we now define a map $C^{k}_{n} :C^\infty_{r}
(G_k)^{\widetilde G_k}\to C^\infty_{r}(G_n)^{\widetilde G_n}$ by
the commutative diagram using Gonzalez' theorem:
$$\xymatrix{C^\infty_{r}(G_k)^{\widetilde G_k} \ar[d]_{\cF_k}\ar[r]^{C^{k}_{n}}&
C^\infty_{r}(G_n)^{\widetilde G_n}\ar[d]^{\cF_n}
 \\
\PW_r^{\rho_k}(\fh_{k,\C})^{\widetilde W_k}\ar[r]_{P^{k}_{n}} & \PW_r^{\rho_n}
(\fh_{n,\C})^{\widetilde W_n} \\
}\, .
$$
\begin{theorem}\label{th-CknSurjective} If $G_k$ propagates $G_n$ and $0<r<R$
then
$$C^{k}_{n} : C^\infty_{r}(G_k)^{\widetilde G_k} \to
C^\infty_{r}(G_n)^{\widetilde G_n}$$ is surjective.
\end{theorem}
\begin{proof} This follows from Theorem \ref{t: Gonzalez} and
Theorem \ref{th-PknSurjective}.
\end{proof}

\begin{theorem} \label{projsys}
Let $r>0$ and assume that $G_k$ propagates $G_n$. Then
the sequences $(\PW_r^{\rho_n} (\fh_{n,\C}^*)^{\widetilde W_n},
P^{k}_{n})$ and $(C^\infty_{r}(G_n)^{\widetilde G_n}, C^{k}_{n})$
form projective systems and
$$
\PW_{r}^{\rho_\infty}(\fh_{\infty,\C})^{\widetilde W_\infty}
  :=\varprojlim \, \PW_r^{\rho_n} (\fh_{n,\C}^*)^{\widetilde W_n}
  \text{ and } C^\infty_{r}(G_\infty )^{\widetilde G_\infty}
  :=\varprojlim \, C_{r} ^\infty (G_n)^{\widetilde G_n}
$$
are nonzero.
\end{theorem}
\begin{proof} This follows from Theorem \ref{th-PknSurjective} and
Theorem \ref{th-CknSurjective}.
\end{proof}
\begin{remark}\label{inf-holo}{\rm
We can view elements $\Phi \in \PW_{r}^{\rho_\infty}(\fh_{\infty,\C})^{\widetilde W_\infty}$ as
holomorphic functions on $\fh_{\infty,\C}^*$ when we view $\fh_{\infty,\C}^*$ as
the spectrum of $\varprojlim  \PW_r^{\rho_n} (\fh_{n,\C}^*)$.
Furthermore, we have a commutative diagram where all maps are surjective
$$\xymatrix{\cdots & C^\infty_{r}(G_n)^{\widetilde G_n}\ar[d]_{\cF_n}&
C^\infty_{r}(G_{n+1})^{\widetilde G_{n+1}}
\ar[d]_{\cF_{n+1}}\ar[l]_{C^{n+1}_{n}}&\ar[l]_{\phantom{XXX} C^{n+2}_{n+1}}
\quad \cdots &
C_{r}^\infty (G_\infty)^{\widetilde G_\infty}\ar[d]_{\cF_\infty}
 \\
\cdots & \PW_r^{\rho_n}(\fh^*_{n,\C})^{\widetilde{W}_n}&  \PW_r^{\rho_{n+1}}(\fh^*_{n+1,\C})^{\widetilde{W}_{n+1}}
 \ar[l]^{P^{n+1}_{ n}}
&\ar[l]^{\phantom{XXXX} P^{n+2}_{n+1}} \quad \cdots & \PW_r^{\rho_\infty}(\fh^{\infty *}_{ \C} )^{\widetilde{W}_\infty}\\
}\, .
$$
}{}\hfill $\diamondsuit$
\end{remark}

\section{Spherical Representations of Compact Groups}\label{sec5}
\setcounter{equation}{0}

\noindent
In the next sections we discuss theorems of Paley-Wiener type for
compact symmetric spaces.  We start by an overview over spherical
representations, spherical functions and
the spherical Fourier transform. Most of the material can be found in
\cite{W2010} and \cite{W2009} but
in part with different proofs. The notation will be as in Section \ref{sec2},
and $G$ or $G_n$ will always stand for a compact group.
In particular,
$M_n=G_n/K_n$ where $G_n$ is a connected compact semisimple Lie group
with Lie algebra $\fg_n$, which for simplicity we
assume is simply connected. The result can easily be
formulated for arbitrary compact symmetric spaces by following the
arguments in \cite{OS}. We will assume that $M_k$ propagates $M_n$.
We denote by $r_k$ and $r_n$ the respective real ranks
of $M_k$ and $M_n$. As always we fix compatible $K_k$-- and
$K_n$--invariant inner products on $\fs_k$ respectively $\fs_n$.

As in Section \ref{sec2}
let $\Sigma_n = \Sigma_n(\mfg_n , \ga_n)$ denote the
system of restricted roots of $\ga_{n,\C}$ in $\mfg_{n,\C}$.
Let $\fh_n$ be a $\theta_n$-stable Cartan subalgebra such
that $\fh_n\cap \fs_n =\fa_n$. Let
$\Delta_n=\Delta (\fg_{n,\C}, \fh_{n,\C})$. Recall
that $\Sigma_n\subset i\ga_n^*$.
We choose positive subsystems
$\Delta_n^+$ and $\Sigma^+_n $ so that
$\Sigma_n^+\subseteqq \Delta_n^+|_{\fa_n}$,
$\Delta_n^+\subseteqq \Delta_k^+|_{\fh_{n,\C}}$,  and
$\Sigma^+_n \subset \Sigma^+_k|_{\ga_n}$.
Consider the reduced root system
$$
\Sigma_{n,2}=\{\alpha\in\Sigma_n\mid 2\alpha\not\in\Sigma_n\}
$$
and its positive subsystem $\Sigma_{n,2}^+ := \Sigma_{n,2} \cap \Sigma^+_n$.
Let
$$
\Psi_{n,2} = \Psi_{2}(\mfg_n , \ga_n) = \{\alpha_{n,1}, \dots , \alpha_{n,r_n}\}
$$
denote the set of simple roots for $\Sigma_{n,2}^+$.
We note the following simple facts; they follow from
the explicit realization (\ref{rootorder}) of the root systems discussed
in \cite[Lemma 1.9]{OW10}.
\begin{lemma} \label{X} Suppose that the $M_n$ are irreducible.
Let $r_n=\dim \fa_n$, the rank of $M_n$.
Number the simple root systems $\Psi_{n, 2}$ as in $(\ref{rootorder})$.
Suppose that $M_k$ propagates $M_n$.
If $j\leqq r_n$ then $\alpha_{k,j}$ is the unique element of
$\Psi_{k,2}$ whose restriction to $\fa_n$ is $\alpha_{n,j}$.
\end{lemma}

Since $M_k$ propagates $M_n$ each irreducible factor of $M_k$ contains
at most  one simple factor of $M_n$.  In particular if $M_n$ is not irreducible
then $M_k$ is not irreducible, but we still can number the simple
roots so that Lemma \ref{X} applies.

We denote the
positive Weyl chamber in $\fa_n$ by $\fa_n^+$ and similarly for
$\fa_k$.  For $\mu \in \gL^+(G_n)$ let
$$
V_{\mu}^{K_n}=\{v\in V_{\mu} \mid
\pi_{\mu}(k)v = v \text{ for all } k\in K_n\}.
$$
We identify $i\fa_n^*$ with $\{\mu \in i\fh_n^*\mid \mu|_{\fh_n\cap \fk_n}=0\}$ and
similar for $\fa_n^*$ and $\fa_{n,\C}^*$.
With this identification in mind set
$$
\Lambda^+(G_n,K_n)
= \left \{\mu\in i\fa_n^* \left |
\tfrac{ (\mu ,\alpha ) }{ ( \alpha ,\alpha )}\in \Z^+ \text{ for all }
\alpha \in \Sigma^+ \right . \right \}.
$$
Most of the time we will simply write  $\Lambda^+_n$ instead of $\Lambda^+(G_n,K_n)$.

Since $G_n$ is connected and $M_n$ is simply connected it follows that
$K_n$ is connected. As $K_n$ is compact there exists a unique  $G_n$--invariant measure
$\mu_{M_n}$ on $M_n$ with $\mu_{M_n}(M_n)=1$.  For brevity we sometimes write $dx$ instead of
$d\mu_{M_n}$.

\begin{theorem}[Cartan-Helgason]\label{t-CH} Assume that $G_n$ is
compact and simply connected. Then the following are equivalent.
\begin{enumerate}
\item $\mu \in \Lambda^+_n$,
\item $\displaystyle{V_{\mu}^{K_n} \ne 0}$,
\item $ \displaystyle{\pi_{\mu}}$ is a subrepresentation of the
representation of  $G_n$  on  $L^2(M_n)$.
\end{enumerate}
When those conditions hold, $\dim V_{\mu}^{K_n} = 1$ and
$\pi_{\mu}$ occurs with multiplicity $1$ in the representation of
$G_n$ on $L^2(M_n)$.
\end{theorem}
\begin{proof} See \cite[Theorem 4.1, p. 535]{He1984}.
\end{proof}

\begin{remark} {\rm If $G_n$  is compact but not simply connected one has
to replace $\Lambda_n^+$  by sub semi--lattices of
weights $\mu$ such that the group
homomorphism $\exp (X)\mapsto e^{\mu (X)}$ is well
defined on the maximal torus
$H_n$, and then the proof of Theorem \ref{t-CH} remains valid.
}\hfill $\diamondsuit$
\end{remark}

Define linear functionals $\xi_{n,j}\in i\ga_n^*$ by
\begin{equation}\label{fundclass1}
\frac{ \langle \xi_{n,i},\alpha_{n,j} \rangle }
{\langle \alpha_{n,j},\alpha_{n,j} \rangle} = \delta_{i,j} \text{ for }
1 \leqq j \leqq r_n\ \ .
\end{equation}
Then for $\alpha \in \Sigma_{n,2}^+$
$$\frac{ \langle  \xi_{n,i},\alpha \rangle  }
{ \langle  \alpha,\alpha \rangle  }\in \Z^+\, .$$
If $\alpha \in \Sigma^+\setminus \Sigma_{n,2}^+$, then
$2\alpha \in \Sigma_{n,2}^+$ and
$$\tfrac{ \langle  \xi_{n,i},\alpha \rangle }
{ \langle  \alpha,\alpha \rangle  }
= 2\tfrac{ \langle  \xi_{n,i},2\alpha \rangle  }
{ \langle 2\alpha,2 \alpha \rangle  }\in \Z^+\, .$$
Hence $\xi_{n,i}\in \Lambda^+_n$. The weights $\xi_{n,j}$ are the
\textit{class 1 fundamental weights for}
$(\mfg_n,\gk_n)$. We set
$$\Xi_n=\{\xi_{n,1},\ldots ,\xi_{n,r_n}\}\, .$$
For $I=(k_1,\ldots ,k_{r_n})\in (\Z^+)^{r_n}$ define
$\mu_I:=\mu(I)=k_1\xi_{n,1}+\ldots +k_{r_n}\xi_{n,r_n}$.
\begin{lemma}\label{le-KnWeights} If $\mu \in i\ga_n^*$ then
$\mu \in \Lambda^+ _n$ if and only if
$\mu =\mu_I$ for some $I\in (\Z^+)^{r_n}$.
\end{lemma}
\begin{proof} This follows directly from the definition of
$\xi_{n,j}$.
\end{proof}
\begin{lemma}\label{le-ResInLambdan} Suppose that $M_k$ is a
propagation of $M_n$. Let
$I_k=(m_1,\ldots ,m_k)\in (\Z^+)^{r_k}$ and
$\mu =\mu_{I_k}$. Then
$\mu|_{\fa_n}\in \Lambda^+_n$. In particular
$\xi_{k,j}|_{\fa_n}\in \Lambda^+ _n$ for $1 \leqq j \leqq r_k$.
\end{lemma}
\begin{proof}
Let $v_\mu \in V_\mu$ be a nonzero highest weight vector and
$e_\mu \in V_\mu$ a $K_k$--fixed unit vector. Denote by
$W=\langle\pi_\mu (G_n)v_\mu \rangle$
the cyclic $G_n$-module generated by $v_\mu $ and let $\mu_n=\mu|_{\ga_n}$.

Write $W=\bigoplus_{j=1}^s W_j$ with $W_j$ irreducible.
If $W_j$ has highest weight $\nu_j \ne \mu$ then $v_\mu  \perp W_j$
so $\langle\pi_\mu (G_n)v_\mu \rangle \perp W_j$, contradicting
$W_j \subset W =\bigoplus W_i$.  Now each $W_j$ has highest weight
$\mu$.  Write $v_\mu =v_1+\ldots +v_s$ with $0 \ne v_j\in W_j$.
As $( v_\mu ,e_{\mu} ) \not= 0$ it follows that
$( v_j,e_{\mu}) \not= 0$ for some $j$. But then the projection
of $e_\mu$ onto $W_j$ is a non-zero $K_n$ fixed vector in $W_j^{K_n}\not= 0$
and hence $\mu |_{\fa_n}\in\Lambda^+_n$.
\end{proof}

\begin{lemma}[\cite{W2010}, Lemma 6]\label{simple-res}
Assume that $M_k$ is a propagation of $M_n$.  If $1 \leqq j \leqq r_n$ then $\xi_{k,j}$
is the unique element of $\Xi_k$ whose restriction of $\ga_n$ is $\xi_{n,j}$.
\end{lemma}

\begin{proof} This is clear when $\fa_k=\fa_n$.
If $r_n<r_k$ it follows from the explicit construction
of the fundamental weights for classical root system; see
\cite[p. 102]{GW1998}.
\end{proof}

\begin{lemma}\label{resmult}
Assume that $\mu_k \in \Lambda^+_k$ is a combination
of the first $r_n$ fundamental weights,  $\mu =
\sum_{j=1}^{r_n} k_j \xi_{k,j}$.
Let
$\mu_n:=\mu|_{\ga_n}=\sum_{j=1}^{r_n}k_j \xi_{n,j}\, $.
If $v$ is a nonzero highest weight vector in $V_{\mu_k}$ then
$\langle\pi_{\mu_k} (G_n)v\rangle$ is irreducible and isomorphic to $V_{\mu_n}$.
Furthermore, $\pi_{\mu_n}$
occurs with multiplicity one in $\pi_{\mu_k}|_{G_n}$.
\end{lemma}
\begin{proof} Each $G_n$--irreducible summand $W$
in $\langle\pi_{\mu_k} (G_n)v\rangle$ has highest weight $\mu_n$.
Fix one such $G_n$--submodule $W$ and let $w \in W$ be a nonzero highest
weight vector. Write $w=w_1+\ldots +w_s$ where each $w_j$ is of some
$\gh_k$--weight $\mu_k - \sum_ik_{j,i}\beta_i$ and where each
$\beta_i$ is a simple root in $\Sigma^+(\mfg_{k},\gh_{k})$
and each $k_{j,i}\in \Z^+$. As $\mu_k|_{\gh_n}=\mu_{n}$ it
follows that
$\langle \sum_{i}k_{j,i}\beta_i|_{\fh_n}, \alpha \rangle = 0$
for all $\alpha \in \Delta (\fg_n,\fh_n)$. Thus
$\sum_{i}k_{j,i}\beta_i|_{\fh_n}=0$. In view of (\ref{rootorder})
each $\langle \beta_i,\alpha_j\rangle \leqq 0$ for
$\alpha_j \in \Delta (\fg_n,\fh_n)$ simple (specifically
$\langle \beta_i,\alpha_j\rangle = 0$ unless $\beta_i = f_{c+1} - f_c$ and
$\alpha_j = f_c - f_{c-1}$, for some $c$, in which case
$\langle \beta_i,\alpha_j\rangle = -1$).  Since every $k_{j,i}\in \Z^+$
now $\langle \beta_i , \alpha_j \rangle = 0$ for each
$\alpha_j \in \Delta (\fg_n,\fh_n)$ simple.  Thus $\beta_i|_{\gh_n}=0$.

Because of the compatibility
of the positive systems $\Delta^+(\mfg_{k,\C},\gh_{k,\C})$ and
$\Delta^+(\mfg_{n,\C},\gh_{n,\C})$ there exists a $\beta \in
\Delta^+(\mfg_{k,\C},\gh_{k,\C})$, $\beta|_{\gh_n}=0$,  such that
$\mu_k -\beta$ is a weight in $V_{\mu_n}$.
Writing $\beta$ as a sum of simple roots, we see that each of the
simple roots
has to vanish on $\ga_n$ and hence the restriction to $\ga_{k}$
can not contain any of the simple roots $\alpha_{k,j}$, $ j=1,\ldots ,r_n$.
But then $\beta $ is perpendicular to the fundamental weights
$\xi_{k,j}$, $j=1,\ldots , r_n$.
Hence $s_\beta (\mu_n- \beta)=\mu_n +\beta$ is also a weight, contradicting
the fact that $\mu_n$ is the highest weight. (Here $s_\beta$ is the reflection
in the hyperplane $\beta =0$.) This shows that $\pi_{\mu_n}$ can
only occur once in $\langle \pi_{\mu_k }(G_n)v\rangle$. In particular,
$\langle \pi_{\mu_k}(G_n)v\rangle$ is irreducible.
\end{proof}

Lemma \ref{resmult} allows us to form direct system of representations, as
follows.  For $\ell \in \N$ denote by $0_\ell = (0,\ldots ,0)$ the zero vector
in $\R^\ell$.  For
$I_{n}=(k_1,\ldots ,k_{r_n})\in (\Z^+)^{r_n}$ let
\begin{equation}\label{I-notation}
\begin{aligned}
\bullet\ &\mu_{I,n}= {\sum}_{j=1}^{r_n}k_j\xi_{n,j}\in \Lambda^+_n;\\
\bullet\ &\pi_{I,n} = \pi_{\mu_{I,n}} \text{ the corresponding spherical
representation};\\
\bullet\ &V_{I,n} = V_{\mu_{I,n}}  \text{ a fixed Hilbert space for the
representation } \pi_{I,n};\\
\bullet\ &v_{I,n} = v_{\mu_{I,n}}  \text{ a highest weight unit vector in }
V_{ I,n};\\
\bullet\ &e_{I,n} = e_{\mu_{I,n}} \text{ a } K_n\text{--fixed unit vector in }
V_{I,n}.
\end{aligned}
\end{equation}

We collect our results in the following Theorem.  Compare
\cite[Section 3]{W2010}.

\begin{theorem}\label{l-inductiveSystemOfRep}
Let $M_k$ propagate $M_n$ and
let $\pi_{I,n}$ be an irreducible representation of $G_n$
with highest weight $\mu_{I,n}\in\Lambda^+_n$.
Let $I_k=(I_n,0_{r_k-r_n})$. Then the following hold.
\begin{enumerate}
\item $\mu_{I,k}\in \Lambda^+ _k$ and
$\mu_{I,k}|_{\ga_n}=\mu_{I,n}$.
\item The $G_n$-submodule of $V_{I,k}$ generated
by $v_{I,k}$ is irreducible.
\item The multiplicity of $\pi_{I,n}$ in
$\pi_{I,k}|_{G_n}$ is $1$, in other words there is an unique
$G_n$--intertwining operator $T_{k}^n: V_{I,n}\to V_{I,k}$ such that
$T_{k}^n(\pi_{I,n} (g)v_{I,n})=\pi_{I,k}(g)v_{I,k}\, .$
\end{enumerate}
\end{theorem}

\begin{remark} {\rm From this point on, when $m\leqq q$ we will always
assume that the Hilbert space $V_{I,m}$ is realized inside $V_{I,q}$ as
$\langle \pi_{I,q}(G_m)v_{I,q}\rangle$.}\hfill $\diamondsuit$
\end{remark}

\section{Spherical Fourier Analysis and the Paley-Wiener Theorem}\label{sec6}

\noindent
In this section we give a short description of
the spherical functions and Fourier analysis on compact
symmetric spaces. Then we state and prove results for
limits of compact symmetric spaces analogous to those
of Section \ref{sec3}.

For the moment let $M=G/K$ be a compact symmetric space. We use
the same notation as in the last section but without the index
$n$. As usual we view functions on $M$ as right $K$--invariant
functions on $G$ via $f(g)=f(g\cdot x_o)$, $x_o=eK$.
For $\mu \in \Lambda^+$ denote by $\deg (\mu )$ the dimension
of the irreducible representation $\pi_\mu$. We note that $\mu \mapsto \deg (\mu )$ extends to a polynomial function on $\fa_\C^*$. Fix a unit $K$-fixed
vector $e_\mu$ and define
\[\psi_\mu (g)=( e_\mu ,\pi_\mu (g)e_\mu) \, .\]
Then $\psi_\mu$ is positive definite spherical function
on $G$, and every positive definite spherical function is
obtained in this way for a suitable representation $\pi$. Define
\begin{equation}\label{def-ell}
\ell^2_d(\Lambda^+ )=
\left \{\{a_\mu \}_{\mu\in\Lambda^+ }\,  \left|\,
a_\mu\in\C\,\,\mathrm{and} \,\,
{\sum}_{\mu\in\Lambda^+}\deg(\mu )|a_\mu |^2<\infty\right.\right \}\, .
\end{equation}
Then $\ell^2_d(\Lambda^+ )$ is a Hilbert space with
inner product
\[((a(\mu))_\mu ,(b(\mu ))_\mu)={\sum}_{\mu\in \Lambda^+} \deg(\mu) a(\mu )\overline{b(\mu )}\, .\]
For $f\in C^\infty (M)$ define the spherical Fourier transform of $f$,
$\cS (f)=\widehat{f} :\Lambda^+ \to \C$ by
\[\widehat{f}(\mu )=(f,\psi_\mu)=\int_M f(g )(\pi_\mu (g)e_\mu ,e_\mu)\, dg
=(\pi_\mu (f)e_\mu ,e_\mu)\]
where $\pi_\mu (f)$ denotes the operator valued Fourier transform of $f$,
$\pi_\mu (f)=\int_G f(g) \pi_\mu (g)\, dg$.
Then the sequence $\cS(f)=(\cS(f)(\mu ))_\mu$ is in $\ell^2_d(\Lambda^+(G,K))$ and
$\|f\|^2=\|\cS (f)\|^2$. Finally, $\cS$ extends by continuity
to an unitary isomorphism
\[\cS : L^2(M)^K\to \ell^2_d(\Lambda^+)\, .\]
We denote by $\cS_\rho$ the map
\begin{equation}\label{def-Srho}
\cS_\rho (f)(\mu )=\cS(f)(\mu -\rho )\, ,\quad \mu\in \Lambda^+ +\rho\, .
\end{equation}
If $f$ is smooth, then $f$ is given by
\[f(x)={\sum}_{\mu\in\Lambda^+}\deg(\mu )\cS (f) (\mu )\psi_\mu (x)=
{\sum}_{\mu\in\Lambda^+} \deg(\mu )\cS_\rho (f)(\mu +\rho )\psi_{\mu}(x)\, .\]
and the series converges in the usual Fr\'echet topology on $C^\infty (M)^K$.
In general, the sum has to be interpreted as an $L^2$ limit.

Let
\[
\Omega:=\{X\in \fa\mid
|\alpha (X)|<\pi/2 \text{ for all } \alpha\in\Sigma\}\, .\]
For $\lambda \in\fa_\C^*$ let $\varphi_\lambda$ denote the
spherical function on the dual symmetric space of noncompact
type $G^d/K$, where the Lie algebra of $G^d$ is given
by $\fg^d:=\fk+i\fs$. Then $\varphi_\lambda$ has a
holomorphic extension as $K_\C$--invariant function
to $K_\C\exp (2\Omega )\cdot x_o\subset
G_\C/K_\C$, cf. \cite[Theorem 3.15]{Opd}, see also
\cite{BOP2005} and \cite{KS2005}. Furthermore
\[\overline{\psi_\mu (x)}=\varphi_{\mu +\rho}(x^{-1})=\varphi_{-\mu -\rho}(x)\]
for $x\in K_\C\exp (2\Omega )\cdot x_o$.
We can therefore define a holomorphic function $\lambda \mapsto \cS_\rho (f)(\lambda )$
by
\begin{equation}\label{HolEx}
\cS_\rho (f)(\lambda )=\int_M f(x)\varphi_{\lambda }(x^{-1})\, dx
\end{equation}
as long as $f$ has support in $K_\C \exp (2\Omega)\cdot x_o$.
$\cS_\rho (f)$ is $W(\fg,\fa )$ invariant and
$\cS_\rho (f)(\mu )=\cS (f)(\mu -\rho )$ for all $\mu \in \Lambda^+(G,K)+\rho$.

Denote by $R$ the injectivity radius of the riemannian exponential map
$\Exp :\fs \to M$.   Following the arguments in  \cite{crit} we get:
\begin{theorem}\label{also-inj-radius}
The injectivity radius $R$ of the classical compact simply connected
riemannian symmetric spaces $M = G/K$, in the riemannian metric given by the
inner product $\langle X,Y\rangle = -\trace(XY)$ on $\fs$, depends
only on the type of the restricted reduced root system
$\Sigma_2(\fg_\C,\ga_\C)$.
It is $\sqrt{2}\,\pi$ for $\Sigma_2(\fg_\C,\ga_\C)$ of type $A$ or $C$ and is
$2\pi$ for $\Sigma_2(\fg_\C,\ga_\C)$ of type $B$ or $D$.
\end{theorem}
\begin{remark}{\em
Since $\Omega$ is given by $|\alpha(X)| < \pi/2$ and the interior of the
injectivity radius disk is given by $|\alpha(X)| < 2\pi$ the set $\Omega$
is contained in the open disk in $\gs$ of center $0$ and radius $R/4$.}
\hfill $\diamondsuit$
\end{remark}

Essentially as before, $\Br$ denotes the closed
metric ball in $M$ with center $x_o$ and radius $r$, and
$C^\infty_r(M)^{\widetilde{K}}$ denotes
the space of $\widetilde{K}$-invariant smooth functions on $M$ supported in $\Br$.

\begin{remark}\label{explain}{\rm
Theorem \ref{t: PW} below is, modulo a $\rho$-shift and $\widetilde{W}$-invariance,
Theorem 4.2 and Remark 4.3 of \cite{OS}. As pointed
out in \cite[Remark 4.3]{OS}, the known value for the
constant $S$ can be different in each part of the theorem.
In Theorem\ref{t: PW}(1) we need that $S<R$ and the closed ball in $\fs$ with
center zero and radius $S$ has to be contained in
$K_\C\exp (i\Omega )\cdot x_o$
to be able to use the estimates from \cite{Opd} for the spherical
functions to show that we actually end up in the Paley-Wiener space.

In Theorem \ref{t: PW}(2) we need only that $S<R$. Thus the constant in
(1) is smaller than the one in (2). That is used in part (3).
For Theorem \ref{t: PW}(4) we also need $\|X\|\leqq \pi/\|\xi_j\|$ for
$j=1,\ldots ,r$.} \hfill $\diamondsuit$
\end{remark}

\begin{theorem}[Paley-Wiener Theorem for Compact Symmetric Spaces]
\label{t: PW} Let the notation be as above. Then the following
hold.
\begin{enumerate}
\item[1.] There exists a constant $S>0$ such that,
for each $0<r<S$ and  $f\in C^\infty_{r}(M)^{\widetilde K}$, the
$\rho$-shifted spherical  Fourier transform $\cS_\rho (f) : \Lambda^+_n+\rho \to\C$
extends to a function in $\PW_r(\ga_\C^*)^{\widetilde W}$.
\item[2.]There exists a constant $S>0$ such that if $F\in\PW_r(\ga_\C^*)^{\widetilde W}$, $0<r<S$,
the function
\begin{equation}\label{defOfIrho}
f(x):={\sum}_{\mu\in\Lambda^+  } \deg (\mu ) F (\mu + \rho )\psi_{\mu} (x)
\end{equation}
is in $C^\infty_{r}(M)^{\widetilde K}$ and
$\cS_\rho {f}(\mu )=F (\mu )$.
\item[3.] For $S$ as in $(1.)$ define $\cI_\rho :
\PW_r(\ga_\C)^{\widetilde W}\to C^\infty_{r}(M)^{\widetilde K}$ by {\rm (\ref{defOfIrho})}.
Then  $\cI_\rho$ is surjective for all $0<r<S$.
\item[4.] There exists a constant $S>0$ such that
for all $0<r<S$ the map $\cS_\rho$ followed by holomorphic extension
defines a bijection $C_{r}(M)^{\widetilde K}\cong \PW_r(\ga_\C )^{\widetilde W}$.
\end{enumerate} \end{theorem}

\begin{proof} This follows from \cite{OS}, (\ref{HolEx}) and Theorem \ref{th-SphericalFctTildeWInv}.
\end{proof}

A weaker version of the following theorem was used in  \cite[Section 11]{OS}.
It used an operator $Q$ which we will define shortly, and
some differentiation,
to prove the surjectivity part of local Paley--Wiener Theorem.
Denote the Fourier transform of $f\in C(G)^G$ by
$\cF (f)$. Recall the
operator $T:\PW_r^\rho(\fh_\C^*)^{\widetilde{W}(\fg,\fh)}\to
\PW_r(\fh_\C^*)^{\widetilde{W}(\fg,\fh)}$ from Theorem \ref{isoPWspaces}. Finally,
for $f\in C (G)$ let $f^\vee (x)=f(x^{-1})$. Then
${}^\vee : C_{r}^\infty (G)^{\widetilde G}\to  C_{r}^\infty (G)^{\widetilde G}$
is a bijection. We will identify $\fa_\C^*$ with the subspace
$\{\lambda\in\fh_\C^*\mid \lambda|_{\fh_\C\cap \fk_\C}=0\}$ without comment
in the following.

\begin{theorem} \label{stronger}
Let $S>0$ be as in {\rm Theorem \ref{t: PW}(1)}
and let $0<r<S$. Then the the
restriction map $\PW_r(\fh^*_\C)^{\widetilde{W}(\fg,\fh)}\to
\PW_r(\fa_\C^*)^{\widetilde{W}(\fg,\fa)}$ is surjective. Furthermore, the map
$C_{r}^\infty (G )^{\widetilde G}\to
C^\infty_{r}(M)^{\widetilde K}$, given by
\[Q(\varphi )(g\cdot x_o )=\int_K \varphi (gk)\, dk,\]
is surjective, and $\cS_\rho \circ Q(f^\vee )= T\circ \cF (f)$ on
$\Lambda^+(G,K)+\rho$.
\end{theorem}

\begin{proof} Surjectivity of the restriction map follows from
Theorem \ref{th-IsoPW1} and Theorem 2.2 in \cite{OW10} stating that $\widetilde W(\fg,\fh)|_\fa=\widetilde W(\fg,\fa)$ and $\rS (\fh)^{\widetilde W(\fg,\fh)}|_\fa=\rS (\fa)^{\widetilde W(\fg,\fa)}$.

Next, we have $Q (\chi_\mu^\vee )(x)=\int_K \chi_\mu (x^{-1}k)\, dk$.
As $\int_K \pi_\mu (k)\, dk$ is the
orthogonal projection onto $V_\mu^K$ it follows that $Q (\chi_\mu^\vee)=0$ if $\mu\not\in \Lambda^+(G,K)$ and
\[Q(\chi_\mu^\vee) (x)  =
(\pi_\mu (x^{-1})e_\mu,e_\mu)=(e_\mu ,\pi_\mu (x)e_\mu )=\psi_\mu (x)\]
for $\mu \in \Lambda^+(G,K)$. Thus, if
$f =\sum_{\mu} \cF (f)(\mu )\chi_\mu $ we have
$$
Q(f^\vee ) (x)= {\sum}_{\mu\in \Lambda^+(G,K)}\cF (f)(\mu )\psi_{\mu }(x)
={\sum}_{\mu\in \Lambda^+(G,K)}\deg(\mu)\tfrac{\cF (f)(\mu )}{\deg(\mu)} \psi_{\mu }(x).
$$
Using the Weyl dimension formula for finite dimensional representations,
$\deg(\mu )=\frac{\varpi (\mu +\rho)}{\varpi (\rho )}$, we get
\[\cS_\rho (Q(f^\vee ))(\mu +\rho)=
\tfrac{\varpi (\mu +\rho)}{\varpi (\rho )}\, \cF (f)(\mu  )
=T(\cF (f))|_\fa (\mu +\rho)\]
for $\mu \in\Lambda^+(G,K)$. Hence $\cS_\rho\circ Q(f^\vee)|_{\Lambda^+(G,K)}
= (T\circ \cF (f)|_{\fa_\C})|_{\Lambda^+(G,K)}$.

Assume that $f\in C^\infty_{r}(G/K)^{\widetilde K}$. Then, by
the Paley-Wiener Theorem, Theorem \ref{t: PW}, there
exists a $\Phi\in\PW_r(\fa_\C^*)^{\widetilde{W}(\fg,\fa)}$ such that
$\Phi =\cS_\rho (f)$ on $\Lambda^+(G,K)$.
Then, by what we just proved, there exists $\Psi\in \PW_r(\fh_\C^*)^{\widetilde{W}(\fg,\fh)}$
such that $\Psi|_{\fa_\C}=\Phi$.
By  Theorem \ref{t: Gonzalez} there
exists $F\in C_{r}(G)^{\widetilde G}$ such that
$T\circ \cF (F)= \Psi$. By the above calculation we have
\[\cS (f) (\mu) =\cS (Q(F^\vee )) (\mu )\quad \text{for all}\quad \mu\in\Lambda^+(G,K)\, .\]
As clearly $Q (F^\vee)$ is smooth, it follows that $Q(F^\vee )=f$ and
hence $Q$ is surjective.
\end{proof}

\section{A $K$-invariant Domain in $M$ and the Projective Limit}\label{sec7}

\noindent
In this section we introduce an $\widetilde{K}$-invariant domain in $\fs$ that behaves well under
propagation of symmetric spaces. We use the notation from \cite{OW10} for the simple roots.

Let $\sigma = 2(\alpha_1+\ldots +\alpha_\ell)$ where the
$\alpha_j\in \Sigma_2^+$ are the simple roots.
For $M$ irreducible let
\begin{equation}\label{eq-Omega1+2}
\begin{aligned}
\Omega^*:=\,\, &\Omega \text{ if $\Sigma_2$ is of type
$A_\ell$ or $C_\ell$}, \\
\Omega^*:= &\bigcap_{w\in W}\{X\in\fa\mid |\sigma (w(X))| < \pi/2\}
\text{ if $\Sigma_2$ is of type $B_\ell$ or $D_\ell$}.
\end{aligned}
\end{equation}
In general, we define $\Omega^*$ to be the product of the $\Omega^*$'s for all
the irreducible factors. Then $\Omega^*$ is a convex Weyl group invariant
polygon in $\fa$.  We also have $\Omega^* = -\Omega^*$. This is easy to
check and in any case will follow from our explicit description of $\Omega^*$.

\noindent
$\mathbf{A_n}$\textbf{:} We have $\fa =\{x\in\R^{n+1}\mid \sum x_j=0\}$,
$n\geqq 1$, and the roots are the $f_i-f_j: x\mapsto x_i-x_j$ for
$i\not= j$. Hence
\begin{equation}\label{eq-OmegaAn}\Omega^*=
    \Omega=\left \{x\in \R^{n+1}\left | \sum x_j=0\,\,\right .
\text{ and } |x_i-x_j| < \tfrac{\pi}{2} \text{ for }
1\leqq i\not= j\leqq n+1\right \}\, .
\end{equation}

\noindent
$\mathbf{B_n}$\textbf{:} We have $\fa=\R^n$, $n\geqq 2$ and $\sigma =
2(f_1+(f_2-f_1)+\ldots + (f_n-f_{n-1}))=2f_n$. The Weyl
group consists of all permutations and sign changes on the $f_i$. Hence
\begin{equation}\label{eq-OmegaBn}
\Omega^*=\{x\in\R^n\mid  |x_j| < \tfrac{\pi}{4} \text{ for } j=1,\ldots ,n\}\, .
\end{equation}

\noindent
$\mathbf{C_n}$\textbf{:} Again $\fa=\R^n$, $n\geqq 3$, and the roots are the
$\pm (f_i\pm f_j)$ and $\pm 2f_j$. If $|x_i|,|x_j|<\pi /4$ then
$|x_i\pm x_j|<\pi /2$. Hence
\begin{equation}\label{eq-OmegaCn}\Omega^*=\Omega=
\{x\in\R^n\mid |x_j|<\tfrac{\pi}{4} \text{ for } j=1,\ldots ,n\}\, .
\end{equation}

\noindent
$\mathbf{D_n}$\textbf{:} Also in this case $\fa=\R^n$ with $n\geqq 4$. We have
$\sigma = 2(f_1+f_2 + (f_2-f_1)+\ldots + (f_n-f_{n-1}))=2(f_2+f_n)$. As the
Weyl group is given by all permutations and even sign changes on the $f_i$,
we get
\begin{equation}\label{eq-OmegaDn}\Omega^*=
\{x\in\R^n \mid |x_i\pm x_j| < \tfrac{\pi}{4} \text{ for }
i,j=1,\ldots ,n\, ,\,\, i\not= j\}.
\end{equation}

\begin{lemma} We have $\Omega^*\subseteqq \Omega$.
\end{lemma}
\begin{proof}
Let $\delta$ be the highest root in $\Sigma^+$. Then
\[\Omega =W\{X\in \overline{\fa^+} \mid \delta (X) < \pi/2\}\, .\]
For the classical Lie algebras, the coefficients
of the simple roots in the highest root are all $1$ or $2$.  Hence
$\Omega^*\subseteqq \Omega$ and the claim follows.
\end{proof}

\begin{remark} \label{leftend} {\rm  The distinction between $\Omega$
and $\Omega^*$ is caused by change in the coefficient in the highest
root of the simple root on the left.  Thus in cases $B_n $ and $D_n$
it goes from $1$ to $2$ as we move up in the rank of $M$:
\begin{eqnarray*}
B_\ell &:&
\setlength{\unitlength}{.5 mm}
\begin{picture}(100,10)
\put(3,2){\circle{2}}
\put(0,5){$1$}
\put(4,2){\line(1,0){23}}
\put(28,2){\circle{2}}
\put(25,5){$2$}
\put(29,2){\line(1,0){13}}
\put(48,2){\circle*{1}}
\put(51,2){\circle*{1}}
\put(54,2){\circle*{1}}
\put(59,2){\line(1,0){13}}
\put(73,2){\circle{2}}
\put(70,5){$2$}
\put(74,2.5){\line(1,0){23}}
\put(74,1.5){\line(1,0){23}}
\put(98,2){\circle*{2}}
\put(95,5){$2$}
\end{picture}
\\
\\
D_\ell &:&
\setlength{\unitlength}{.5 mm}
\begin{picture}(100,10)
\put(3,2){\circle{2}}
\put(0,5){$1$}
\put(4,2){\line(1,0){23}}
\put(28,2){\circle{2}}
\put(25,5){$2$}
\put(29,2){\line(1,0){13}}
\put(48,2){\circle*{1}}
\put(51,2){\circle*{1}}
\put(54,2){\circle*{1}}
\put(59,2){\line(1,0){13}}
\put(73,2){\circle{2}}
\put(68,5){$2$}
\put(74,1.5){\line(2,-1){13}}
\put(88,-5){\circle{2}}
\put(91,-7){$1$}
\put(74,2.5){\line(2,1){13}}
\put(88,9){\circle{2}}
\put(91,7){$1$}
\end{picture}
\end{eqnarray*}
while in cases $A_n$ and $C_n$ it doesn't change:
\begin{eqnarray*}
A_\ell&:&
\setlength{\unitlength}{.5 mm}
\begin{picture}(100,15)
\put(3,2){\circle{2}}
\put(0,5){$1$}
\put(4,2){\line(1,0){23}}
\put(28,2){\circle{2}}
\put(25,5){$1$}
\put(29,2){\line(1,0){23}}
\put(53,2){\circle{2}}
\put(50,5){$1$}
\put(54,2){\line(1,0){13}}
\put(72,2){\circle*{1}}
\put(75,2){\circle*{1}}
\put(78,2){\circle*{1}}
\put(84,2){\line(1,0){13}}
\put(98,2){\circle{2}}
\put(95,5){$1$}
\end{picture}\\
\\
C_\ell &:&
\setlength{\unitlength}{.5 mm}
\begin{picture}(100,10)
\put(3,2){\circle*{2}}
\put(0,5){$2$}
\put(5,2){\line(1,0){23}}
\put(28,2){\circle*{2}}
\put(28,5){$2$}
\put(29,2){\line(1,0){13}}
\put(48,2){\circle*{1}}
\put(51,2){\circle*{1}}
\put(54,2){\circle*{1}}
\put(59,2){\line(1,0){13}}
\put(73,2){\circle*{2}}
\put(70,5){$2$}
\put(74,2.5){\line(1,0){23}}
\put(74,1.5){\line(1,0){23}}
\put(98,2){\circle{2}}
\put(95,5){$1$}
\end{picture}
\end{eqnarray*}
}\hfill $\diamondsuit$
\end{remark}

\begin{lemma}\label{le-Omega*} If $S>0$ such that
$\{X\in \fs \mid \|X\|\leqq S\}\subset \Ad (K) \Omega^*\}$, then
we can use $S$ as the constant in {\rm Theorem \ref{t: PW}(1)}.
\end{lemma}
\begin{proof} Recall from \cite[Remark 4.3]{OS} that
Theorem \ref{t: PW}(1) holds when $0 < S < R$ and
\begin{equation}\label{eq-Omega}
\{X\in \fs \mid \|X\|\leqq S\}\subseteqq \Ad (K) \Omega\, .
\end{equation}
But $\Ad (K)\Omega$ is open in $\fs$, and $\Exp : \Ad (K)\Omega \to M$ is
injective by Theorem \ref{also-inj-radius}.
Hence, if (\ref{eq-Omega}) holds then $S<R$, and the claim follows from
the first part of Remark \ref{explain}.
\end{proof}

We will now apply this to sequences $\{M_n\}$ where $M_k$ is a propagation of
$M_n$ for $k \geqq n$. We use the same notation as before and add the index
$n$ (or $k$) to indicate the dependence of the space $M_n$ (or $M_k$).
We start with the following lemma.
\begin{lemma}\label{le-Omega2*}  If $k\geqq n$ then
$\Omega_n^*=\Omega^*_k\cap \fa_n$.
\end{lemma}
\begin{proof} We can assume that $M$ is irreducible. As $M_k$ propagates
$M_n$ it follows that we are only adding simple roots to the left
on the Dynkin diagram for $\Sigma_2$. Let $r_n$ denote the rank of $M_n$ and
$r_k$ the rank of $M_k$. We can assume that $r_n < r_k$, as the claim is obvious
for $r_n = r_k$. We use the above explicit description $\Omega^*$ given above
and case by case
inspection:

Assume that $\Sigma_{n,2}$ is of type $A_{r_n}$ and $\Sigma_{k,2}$ is
of type $A_{r_k}$ with $r_n < r_k$. It follows from (\ref{eq-OmegaAn}) that
$\Omega_n^*\subseteqq \Omega_k^*\cap \fa_n$.
Let $(0,x)\in \Omega_n^*$.
For $j>i$ we have
\begin{equation}\label{eq-alpha(0,x)}
\pm (f_{j}-f_i) ((0,x))=\left\{\begin{array}{c@{\quad\text{for}\quad}l}
\pm (x_j-x_i)&j\leqq r_n+1\\
\mp (-x_i)&j>r_n+1\geqq i\\
0&j,i>r_n+1
\end{array}\right.
\end{equation}
Let $i\leqq r_n+1$. Then, using that $x_i=-\sum_{j\not=i}x_j$ and
$|x_i-x_j| <\pi/2$, we get
$$
-r_k\tfrac{\pi}{2} <   \sum_{i\ne j} (x_i - x_j) = r_k x_i-\sum_{j\not= i}x_j
=(r_k+1)x_i < r_k\tfrac{\pi}{2}\, .
$$
Hence
\[-\tfrac{\pi}{2}<-\tfrac{r_k}{r_k+1}\, \tfrac{\pi}{2}<x_i<\tfrac{r_k}{r_k+1}\, \tfrac{\pi}{2}<\tfrac{\pi}{2}\, .\]
It follows now from (\ref{eq-alpha(0,x)}) that $(0,x)\in \Omega_k^*\cap \fa_n$.

The cases of types $B$ and $C$ are obvious from (\ref{eq-OmegaBn}) and
(\ref{eq-OmegaCn}). For the case of type $D$
we note that $|x_i\pm x_j|< \tfrac{\pi}{4}$ implies both
-$\tfrac{\pi}{4} < x_i - x_j < \tfrac{\pi}{4}$ and
-$\tfrac{\pi}{4} < x_i + x_j < \tfrac{\pi}{4}$.  Adding,
-$\tfrac{\pi}{2} < 2x_i < \tfrac{\pi}{2}$, so $|x_i| < \tfrac{\pi}{4}$.
Hence $(0,x)\in \Omega_k^*\cap \fa_n$ if and only if
$x\in \Omega_n^*$ by (\ref{eq-OmegaDn}).
\end{proof}

We can now proceed as in Section \ref{sec3}. We will always assume that
$S>0$ is small enough that $\Omega^*$ contains the closed ball in
$\fs$ of radius $S$.  Define
$C^{k}_{n}:C^\infty_{r}(M_k)^{\widetilde K_k}\to C^\infty_{r}(M_n)^{\widetilde  K_n}$ by
$C^{k}_{n}:=\cI_{n,\rho_n} \circ P^{k}_{n}\circ \cS_{k, \rho_k}$, in other words
\[C^{k}_{n}(f)(x)=\sum_{I\in(\Z^+)^{r_n}} \deg(\mu_{I,n})\widehat{f}
(\mu_{I,k}-\rho_k+\rho_n) \psi_{\mu_{I,n}}(x)\, .\]

\begin{theorem}[Paley-Wiener Isomorphism-II]\label{th-PWcompactII} If
$M_k$ propagates $M_n$ and $0<r<S$ then
\begin{enumerate}
\item the map $P^{k}_{n}:
\PW_r (\fa_{k,\C}^*)^{\widetilde  W_k} \to
\PW_r(\fa_{n,\C}^*)^{\widetilde  W_n}$ is surjective, and
\item the map $C^{k}_{n} :C^\infty_{r}(M_k)^{\widetilde  K_k}
\to C^\infty_{r}(M_n)^{\widetilde K_n}$ is surjective.
\end{enumerate}
\end{theorem}
\begin{proof} This follows from  Theorem \ref{th-IsoPW1}, Lemma \ref{le-Omega*},
and Lemma \ref{le-Omega2*}.
\end{proof}

We now assume that $\{M_n,\iota_{k,n}\}$ is a injective system of
riemannian symmetric spaces of compact type
such that the direct system maps $\iota_{k,n}: M_n\to M_k$ are injections
and $M_k$ is a propagation of $M_n$ along a cofinite subsequence.
Passing to that cofinite subsequence we may assume that
$M_k$ is a propagation of $M_n$ whenever $k \geqq n$.  Denote
$M_\infty =\varinjlim M_n\,$.

The compact symmetric spaces of Table \ref{symmetric-case-class} give
rise to the following injective limits of symmetric spaces.

\begin{equation}\label{e-infiniteDim}
\begin{aligned}
{\rm 1.}\ &\bigl (\SU (\infty) \times \SU(\infty)\bigr )/\diag\,\SU(\infty),
    \text{ group manifold } \SU(\infty),\\
{\rm 2.}\ &\bigl (\Spin(\infty)\times \Spin(\infty)\bigr )/\diag\, \Spin(\infty),
    \text{ group manifold } \Spin(\infty),\\
{\rm 3.}\ &\bigl (\Sp(\infty)\times \Sp(\infty)\bigr )/\diag\, \Sp(\infty),
    \text{ group manifold } \Sp (\infty),\\
{\rm 4.}\ &\SU (p + \infty)/\mathrm{S}(\U(p)\times \U(\infty)),\
    \C^p \text{ subspaces of } \C^\infty, \\
{\rm 5.}\ &\SU (2\infty)/[\mathrm{S}(\U (\infty) \times \U (\infty))],\
    \C^\infty \text{ subspaces of infinite codim in } \C^\infty, \\
{\rm 6.}\ &\SU (\infty)/\SO (\infty),\
    \text{ real forms of } \C^\infty \\
{\rm 7.}\ &\SU (2\infty)/\Sp (\infty), \
    \text{ quaternion vector space structures on } \C^\infty,\\
{\rm 8.}\ &\SO (p + \infty)/[\SO (p)\times \SO (\infty)],\text{ oriented }
    \R^p \text{ subspaces of } \R^\infty, \\
{\rm 9.}\ &\SO (2\infty)/[\SO (\infty)\times \SO (\infty)], \
    \R^\infty \text{ subspaces of infinite codim in } \R^\infty, \\
{\rm 10.}\ &\SO (2\infty)/\U (\infty), \
    \text{ complex vector space structures on } \R^\infty, \\
{\rm 11.}\ &\Sp (p + \infty)/[\Sp(p)\times \Sp(\infty)],\
    \H^p \text{ subspaces of } \H^\infty, \\
{\rm 12.}\ &\Sp (2\infty)/[\Sp(\infty)\times \Sp(\infty)], \
    \H^\infty \text{ subspaces of infinite codim in } \H^\infty, \\
{\rm 13.}\ &\Sp (\infty)/\U (\infty), \
    \text{ complex forms of } \H^\infty.
\end{aligned}
\end{equation}

We also have as before injective systems
$\fg_n\hookrightarrow \fg_k$, $\fk_n\hookrightarrow \fk_k$, $\fs_n\hookrightarrow \fs_k$,
and $\fa_n\hookrightarrow \fa_k$ giving rise to corresponding injective systems. Let
$$
\mfg_\infty :=\varinjlim \mfg_n\, ,\quad \gk_\infty :=\varinjlim \gk_n\, ,\quad
\gs_\infty :=\varinjlim\gs_n\, , \quad \ga_\infty := \varinjlim\ga_n\, ,
 \quad\text{and}, \quad \fh_\infty :=\varinjlim \fh_n\, .$$
Then $\mfg_\infty=\gk_\infty \oplus \gs_\infty$ is the eigenspace decomposition of
$\mfg_\infty $ with respect to the involution $\theta_\infty :=\varinjlim \theta_n$,
$\fa_\infty$ is a maximal abelian subspace of $\fs_\infty$.

Further, we have also projective systems
$\{\PW_{r}(\fa_{n,\C})^{\widetilde W_n}\}$
and $\{C_{r}(M_n)^{\widetilde K_n}\}$ with surjective projections, and
their limits.
$$
\PW_r(\fa^*_{\infty,\C} )^{\widetilde W_\infty}:=
   \varprojlim \PW_{r}(\fa_{n,\C}^*)^{\widetilde W_n} \text{ and }
   C_{r}(M_\infty)^{\widetilde K_\infty} :=
   \varprojlim C_{r}(M_n)^{\widetilde K_n} \, .
$$
As before we view the elements of
$\PW_r(\fa^*_{\infty,_\C})^{\widetilde W_\infty}$ as
$\widetilde W_\infty$--invariant functions on $\fa_{\infty,\C}^*$. For $\mathbf{f}=(f_n)_n\in
C_{r}(M_\infty)^{\widetilde K_\infty}$ define $\cS_{\rho,\infty} (\mathbf{f})
\in
\PW_r(\fa_{\infty,\C}^*)^{\widetilde W_\infty}$ by
\begin{equation}
\cS_{\rho,\infty} (\mathbf{f}):=\{\cS_{\rho,n}(f_n)\}\, .
\end{equation}
Then $\cS_{\rho,\infty} (\mathbf{f})$ is well defined by Theorem \ref{th-PWcompactII}
and we have a commutative diagram
$$\xymatrix{\cdots & C^\infty_{r}(M_n)^{\widetilde K_n}\ar[d]_{\cS_{\rho,n}}&
C^\infty_{r}(M_{n+1})^{\widetilde K_{n+1}}\ar[d]_{\cS_{\rho,n+1}}\ar[l]_(0.55){C^{n+1}_{n}}&
\ar[l]_(0.3){C^{n+2}_{n+1}}\qquad \cdots &
C_{r}(M_\infty)^{\widetilde  K_\infty}\ar[d]_{\cS_{\rho,\infty}}
 \\
\cdots & \PW_r(\fa^*_{n,\C})^{\widetilde W_n}&  \PW_r(\fa^*_{n+1,\C})^{\widetilde W_{n+1}} \ar[l]^(0.55){P^{n+1}_{n}}
&\ar[l]^(0.3){P^{n+2}_{n+1}} \qquad \cdots & \PW_r(\fa_{\infty,\C}^* )^{\widetilde W_\infty}\\
}
$$
Also see \cite{OW,KW2009} for the spherical Fourier transform and
direct limits.

\begin{theorem}[Infinite dimensional Paley-Wiener Theorem-II]
\label{th-ProjLimCompact}
In the above notation,
$\PW_r(\fa_{\infty,\C}^*)^{\widetilde W_\infty}\not=\{0\}$,
$C_{r }(M_\infty)^{\widetilde K_\infty}\not=\{0\}$,
and the spherical Fourier transform
\[\cF_\infty :C_{r}(M_\infty)^{\widetilde K_\infty}
\to \PW_r(\fa^*_{\infty,\C})^{\widetilde W_\infty}\]
is injective.
\end{theorem}

\section{Comparison with the $L^2$ Theory}\label{sec8}
\noindent
Theorem \ref{th-ProjLimCompact} is based on limits of $C^\infty$  and
$C^\infty_c$ spaces, rather than isometric immersions, $L^2$ spaces, and
unitary representation theory.  Just as the $L^2$ space of a
compact symmetric space is the Hilbert space completion of the
corresponding $C^\infty$ space, it is now known \cite[Proposition 3.27]{W2011}
that the same is true for inductive limits of compact symmetric spaces.
Here we discuss those inductive limit $L^2$ spaces, clarifying the
connection between Paley--Wiener theory and $L^2$ Fourier transform
theory.

Any consideration of the
projective limit of $L^2$ spaces follows similar lines by replacing
the the maps of the inductive limit by the corresponding orthogonal
projections, because inductive and projective limits are the
same in the Hilbert space category.

The material of this section is taken from \cite[Section 3]{W2010}
and \cite[Section 3]{W2011}
and adapted to our setting.  We assume without further comments that
all extensions are propagations.

There are three steps to the comparison.  First, we describe the construction
of a direct limit Hilbert space
$L^2(M_\infty) := \varinjlim \{L^2(M_n),L_{m,n}\}$ that carries a natural
multiplicity--free unitary action of $G_\infty$.  Then we describe the
ring $\cA(M_\infty) := \varinjlim \{\cA(M_n),\nu_{m,n}\}$
of regular functions on $M_\infty$ where $\cA(M_n)$ consists of the
finite linear combinations of the matrix coefficients of the $\pi_\mu$
with $\mu \in \Lambda_n^+(G_n,K_n)$ and such that $\nu_{m,n}(f)|_{M_n} = f$.
Thus $\cA(M_\infty)$ is a (rather small) $G_\infty$--submodule of the projective
limit $\varprojlim\{\cA(M_n),\text{ restriction}\}$.  Third, we describe
a $\{G_n\}$--equivariant morphism
$\{\cA(M_n),\nu_{m,n}\} \rightsquigarrow \{L^2(M_n),L_{m,n}\}$ of
direct systems that embeds $\cA(M_\infty)$ as a dense $G$--submodule of
$L^2(M_\infty)$, so that $L^2(M_\infty)$ is $G_\infty$--isomorphic to a Hilbert
space completion of the function space $\cA(M_\infty)$.

We recall first some basic facts about the vector valued Fourier transform on
$M_n$ as well as the decomposition of $L^2(M_n)$ into irreducible summands.
To simplify notation write $\Lambda^+_n$ for $\Lambda^+(G_n,K_n)$.
Let $\mu \in\Lambda_n^+$ and let $V_{n,\mu}$ denote the irreducible
$G_n$--module of highest weight $\mu$. Recursively in $n$, we
choose a highest weight vector $v_{n,\mu} \in V_{n,\mu}$ and
and a $K_n$--invariant unit vector $e_{n,\mu} \in V_\mu^{K_n}$
such that
(i) $V_{n-1,\mu} \hookrightarrow V_{n,\mu}$ is isometric and
$G_{n-1}$--equivariant and sends $v_{n-1,\mu}$ to a multiple
of $v_{n,\mu}$, (ii) orthogonal projection $V_{n,\mu} \to V_{n-1,\mu}$
sends $e_{n,\mu}$ to a non--negative real
multiple $c_{n,n-1,\mu}e_{n-1,\mu}$ of $e_{n-1,\mu}$, and (iii)
$\langle v_{n,\mu},e_{n,\mu}\rangle =1$.
(Then $0 < c_{n,n-1,\mu} \leqq 1$.)  Note that orthogonal projection
$V_{m,\mu} \to V_{n,\mu}, m \geqq n$, sends
$e_{m,\mu}$ to $c_{m,n,\mu}e_{n,\mu}$ where
$c_{m,n,\mu} = c_{m,m-1,\mu}\cdots c_{n+1,n,\mu}$.

The Hermann Weyl degree formula provides polynomial functions on $\ga_\C^*$
that map $\mu$ to $\deg(\pi_{n,\mu}) = \dim V_{n,\mu}$.
Earlier in this paper we had written
$\deg(\mu)$ for that degree when $n$ was fixed, but here it is crucial to
track the variation of $\deg(\pi_{n,\mu})$ as $n$ increases.
Define a map $v\mapsto f_{n,\mu, v}$ from $V_{n,\mu}$ into
$L^2(M_n)$ by
\begin{equation}\label{def-Coeff}
f_{n,\mu,v}(x) = \langle v,\pi_{n,\mu}(x)e_\mu \rangle \,.
\end{equation}
It follows by the Frobenius--Schur orthogonality relations that
$v\mapsto \deg(\pi_{n,\mu} )^{1/2} f_{\mu,v}$ is a unitary
$G_n$ map from $V_\mu $ onto its image in $L^2(M_n)$.

The operator valued Fourier transform
$$L^2(G_n) \to \bigoplus_{\mu\in \Lambda^+_n}\mathrm{Hom} (V_{n,\mu},V_{n,\mu})
\cong \bigoplus_{\mu\in \Lambda^+_n}V_{n,\mu}\otimes V_{n,\mu}^*$$
is defined by $f \mapsto \bigoplus_{\mu\in \Lambda^+_n}\pi_{n,\mu} (f)$ where
$\pi_{n,\mu} (f) \in \mathrm{Hom} (V_{n,\mu},V_{n,\mu})$ is given by
\begin{equation}
\pi_{n,\mu} (f) v := \int_{G_n} f(x) \pi_{n,\mu} (x)v\,\text{ for } f\in L^2(G_n)\,  .
\end{equation}
Denote by $P^{K_n}_\mu $ the orthogonal projection
$V_{n,\mu} \to V_{n,\mu} ^{K_n}$. Then
$P^{K_n}_\mu (v)=\int_{K_n}\pi_{n,\mu} (k)v\, dk$,
and if $f$ is right $K_n$--invariant, then
$$\pi_{n,\mu} (f)=\pi_{n,\mu} (f)P^{K_n}_\mu\, .$$
That gives us the vector valued Fourier transform
$f\mapsto \widehat{f}: \Lambda^+_n\to \bigoplus_{\mu\in\Lambda_n^+} V_{n,\mu}$\,,
\begin{equation}
L^2(M_n)\to \bigoplus_{\mu\in\Lambda_n^+} V_{n,\mu}
\text{ defined by } f\mapsto \widehat{f}(\mu ) :=\pi_{n,\mu} (f)e_{n,\mu}\, .
\end{equation}
Then the Plancherel formula for $L^2(M_n)$ states that
\begin{equation}\label{e-FourierExpansion}
f=\sum_{\mu\in \Lambda_n^+} \deg(\pi_{n,\mu} ) f_{\mu ,\widehat{f}(\mu )}
=\sum_{\mu\in \Lambda_n^+} \deg(\pi_{n,\mu} ) \langle
\widehat{f}(\mu ), \pi_{n,\mu} (\, \cdot \, )e_{n,\mu} \rangle
\end{equation}
in $L^2(M_n)$ and
\begin{equation}\label{e-Norm}
\|f\|^2_{L^2}=\sum_{\mu\in \Lambda^+_n} \deg(\pi_{n,\mu} )
\|\widehat{f}(\mu )\|_{HS}^2\, .
\end{equation}
If $f$ is smooth, then the series in (\ref{e-FourierExpansion}) converges
in the $C^\infty$ topology of $C^\infty (M_n)$.

For $n\leqq m$ and $\mu = \mu_{I,n}\in \Lambda^+_n$ consider the following
diagram of unitary $G_n$-maps, adapted from \cite[Equation 3.21]{W2011}:
$$
\xymatrix{V_{\mu_{I,n}}
\ar[d]_{v\mapsto \deg(\pi_{n,\mu})^{1/2} f_{\mu_{I,n},v}}\ar[r]^{v\mapsto v}
& V_{\mu_{I,m}}\ar[d]^{v\mapsto \deg(\pi_{m,\mu})^{1/2}  f_{\mu_{I,m},v}}\\
L^2(M_n) \ar[r]_{L_{m,n}}& L^2(M_m)   }
$$
where $L_{m,n}: L^2(M_n)\to L^2(M_k)$ is the $G_n$--equivariant
partial isometry defined by
\begin{equation}\label{def-psi(m,n)}
L_{k,n}: \sum_{I_n} f_{\mu_{I,n},w_{I}}\mapsto
 \sum_{I_m} c_{m,n,\mu}\sqrt{\tfrac{\deg (\pi_{m,\mu})}{\deg (\pi_{n,\mu})}}\,
f_{\mu_{I,m},w_{I}}\, ,\quad w_I\in V_{n,\mu}\, .
\end{equation}
As in \cite[Section 4]{W2011} we have
\begin{theorem}\label{fun-restriction}
The map $L_{k,n}$ of {\rm (\ref{def-psi(m,n)})}
is a $G_n$--equivariant partial isometry with image
$$
\Im (L_{m,n} )\cong {\bigoplus}_{I\in (\Z^+)^{r_k},\
k_{r_n+1}=\ldots = k_{r_k}=0} V_{\mu_{I}}\, .
$$
If
$n\leqq m \leqq k$ then
$$L_{k,n}=L_{m,n}\circ L_{k,m}$$
making $\{L^2(M_n),L_{k ,n}\}$ into
a direct system of Hilbert spaces.
\end{theorem}

\noindent
Define
\begin{equation}
L^2(M_\infty ) := \varinjlim L^2(M_n),
\end{equation}
direct limit in the category of Hilbert spaces and unitary injections.

From construction of the $L_{m,n}$ we now have
\begin{theorem}[\cite{W2010}, Theorem 13] \label{cor-symm-mfree}
The left regular representation of $G_\infty$ on $L^2(M_\infty)$
is a multiplicity free discrete direct sum of
irreducible representations.  Specifically, that left regular
representation is $\sum_{I \in \cI} \pi_I$ where
$\pi_I = \varinjlim \pi_{I,n}$ is the irreducible representation
of $G_\infty$ with highest weight $\xi_I := \sum k_r\xi_r$.  This
applies to all the direct systems of {\rm (\ref{e-infiniteDim})}.
\end{theorem}

The problem with the partial isometries $L_{m,n}$ is that they do not
work well with restriction of functions, because of rescaling
and because $L_{m,n}(L^2(M_n)^{K_n}) \not\subset L^2(M_m)^{K_m}$
for $n < m$.
In particular the spherical functions $\psi_{I,n}(g) :=
\langle e_{I,n} ,\pi_{I,n}(g)e_{I,n})\rangle$ do not map
forward, in other words $L_{m,n}(\psi_{I,n})\not=\psi_{I,m}$.

We deal with this by viewing $L^2(M_\infty)$ as a Hilbert space completion
of the ring $\cA(M_\infty) := \varinjlim \cA(M_n)$ of regular functions on
$M_\infty$.  Adapting \cite[Section 3]{W2011} to our notation, we define
\begin{equation}\label{symm-quo-lim-coef}
\begin{aligned}
&\cA(\pi_{n,\mu})^{K_n} = \{\text{finite linear combinations of the }
  f_{\mu,I_n,w_I}
  \text{ where } w_I \in V_{n,\mu} \}, \\
&\nu_{m,n,\mu}:\cA(\pi_{n,\mu})^{K_n} \hookrightarrow
  \cA(\pi_{m,\mu})^{K_m} \text{ by }
        f_{\mu,I_n,w_I} \mapsto f_{\mu,I_m,w_I}\,\, .
\end{aligned}
\end{equation}
Thus \cite[Lemma 2.30]{W2011} says that if
$f \in \cA(\pi_{n,\mu})^{K_n}$ then $\nu_{m,n,\mu}(f)|_{M_n} = f$.

The ring of regular functions on $M_n$ is
$\cA(M_n) := \cA(G_n)^{K_n}
= \sum_\mu \cA(\pi_{n,\mu})$, and the $\nu_{m,n,\mu}$ sum to
define a direct system $\{\cA(M_n),\nu_{m,n}\}$.  Its limit is
\begin{equation}\label{symm-quo-lim-reg}
\cA(M_\infty) := \cA(G_\infty)^{K_\infty} =
\varinjlim \{\cA(M_n),\nu_{m,n}\}.
\end{equation}
As just noted, the maps of the direct system $\{\cA(M_n),\nu_{m,n}\}$
are inverse to restriction of functions, so $\cA(M_\infty)$ is a
$G_\infty$--submodule
of the inverse limit $\varprojlim \{\cA(M_n), \text{ restriction}\}$.

For each $n$,
$\cA(M_n)$ is a dense subspace of $L^2(M_n)$ but, because the
$\nu_{m,n}$ distort the Hilbert space structure,
$\cA(M_\infty)$ does not sit naturally as a subspace of $L^2(M_\infty)$.
Thus we use the $G_n$--equivariant maps
\begin{equation}\label{sym-rel-map-sys}
\eta_{n,\mu}: \cA(\pi_{n,\mu})^{K_n} \to
\cH_{\pi_n}\widehat{\otimes}(w_{n,\mu^*}\C) \text{ by }
f_{\mu,I_n,w_I} \mapsto c_{n,1,\mu}
\sqrt{\deg \pi_{n,\mu}}\, f_{\mu,I_n,w_I}.
\end{equation}
where $c_{m,n,\mu}$ is the length of the projection of $e_{m,\mu}$
to $V_{n,\mu}$.  Now \cite[Proposition 3.27]{W2011} says
\begin{proposition}\label{sym-quo-comparison}
The maps $L_{m,n,\mu}$ of {\rm (\ref{def-psi(m,n)})},
$\nu_{m,n,\mu}$ of {\rm (\ref{symm-quo-lim-coef})} and
$\eta_{n,\mu}$ of {\rm (\ref{sym-rel-map-sys})}
satisfy $$(\eta_{m,\mu}\circ \nu_{m,n,\mu})(f_{\mu,I_n,w_I}) =
(L_{m,n,\mu}\circ\eta_{n,\mu})(f_{\mu,I_n,w_I})$$
for $f_{u,v,n} \in \cA(\pi_{n,\mu})^{K_n}$.  Thus they inject
the direct system $\{\cA(M_n), \nu_{m,n}\}$
into the direct system $\{L^2(M_n),L_{m,n}\}$.
That map of direct systems defines a $G_\infty$--equivariant injection
$$
\widetilde{\eta}: \cA(M_\infty) \to L^2(M_\infty)
$$
with dense image.  In particular $\eta$ defines a pre Hilbert space
structure on $\cA(M_\infty)$ with completion isometric to $L^2(M_\infty)$.
\end{proposition}
This describes $L^2(M_\infty)$ as an ordinary Hilbert space completion of
a natural function space on $M_\infty$.

\end{document}